\documentclass[]{siamltex}
\usepackage{amsmath}
\usepackage{amssymb}
\usepackage{amsfonts}
\usepackage{texdraw}
\usepackage{graphicx}
\usepackage{epsfig}
\usepackage{color}
\newtheorem{acknowledgement}{Acknowledgement}[section]

\newtheorem{algorithm}{Algorithm}[section]

\newtheorem{remark}{Remark}[section]

\newtheorem{example}{Example}[section]

\newcommand{\od}{\frac{d}{dt}}
\newcommand{\pd}{{\partial }}

\newcommand {\opLip}{\operatorname{Lip}}
\newcommand {\BV}{\operatorname{BV}}

\newcommand {\divp}{\operatorname{div}^{+}}
\newcommand {\divm}{\operatorname{div}^{-}}
\newcommand {\Dt}{\Delta t}
\newcommand {\Ph}{\operatorname{P}_h}

\newcommand {\CIh}{\mathcal{I}^h}

\newcount\refnum\refnum=0
\def\myref{{\global\advance\refnum by 1} {\bf \large Lecture \the \refnum. }}

\begin{document}

\bibliographystyle{plainnat}
\title{Convergence Analysis of a Finite Difference Scheme \\
for the Gradient Flow associated with the ROF Model}
\author{Qianying Hong \footnote{qyhong@math.ku.edu. This author is associated with Department of Mathematics, University of Kansas, Lawrenceville,  Kansas 66045.}, 
Ming-Jun Lai \footnote{This author is associated with Department of Mathematics, University of Georgia, Athens, GA 30602. His email address is mjlai@math.uga.edu} \and
Jingyue Wang\footnote{jwang@math.ku.edu. This author is associated with Department of Mathematics, University of Kansas, Lawrenceville,  Kansas 66045.}}
%Department of Mathematics\\
%The University of Georgia\\
%Athens, GA 30602.}
\maketitle

\begin{abstract}
We present a convergence analysis of a finite difference scheme for the time dependent partial
different equation called gradient flow associated with the Rudin-Osher-Fatemi model. We
devise an iterative algorithm to compute the solution of the finite difference scheme and prove the
convergence of the iterative algorithm. Finally  computational experiments are shown to demonstrate the convergence
of the finite difference scheme. An application for image denoising is given.  
\footnotesize{This is a version of Jan. 2012. }
\end{abstract}

% \noindent
% {\bf Key Words and Phrases: } Finite Difference Method, Evolution Equation,  Gradient Flow, Gradient and TV Flows

% \noindent
% {\bf Mathematics Subject Classification 2000:}  15A09, 41A15, 65F10, 90C26

% \noindent
% {\bf Short Title:} An Unconstrained  $\ell_q$ Minimization with $0<q\le 1$

\section{Introduction}
The well-known ROF model may be approximated in the following way
\begin{equation}
\label{ROFv2}
\min_{u \in \BV(\Omega)} \int_\Omega \sqrt{\epsilon + |\nabla u|^2}dx  + \frac{1}{2\lambda}
\int_\Omega |u- f|^2 dx.
\end{equation}
As $\epsilon>0$, the above minimizing functional is differentiable. Thus, the Euler-Lagrange equation
associated with the above minimization is
\begin{equation}
\hbox{div} \left( \frac{\nabla u}{\sqrt{\epsilon + |\nabla u|^2}}\right) - \frac{1}{\lambda} (u- f)=0.
\end{equation}
Solution of this partial differential equation can be further approximated. Let  us consider the time
evolution version of the PDE:
\begin{equation}
\label{epsTEROF2}
\begin{cases}
\od u = \hbox{ div }\left( \dfrac{\nabla u}{\sqrt{\epsilon + |\nabla u|^2}}
\right) - \frac{1}{\lambda}(u- f)  & \in \Omega_T\cr
{\partial \over \partial {\bf n}} u = 0  &\hbox{ on } \pd \Omega_T\cr
u(\cdot,0)=u_0(\cdot), & \Omega,
\end{cases}
\end{equation}
where $f$ is given a noised image, $\Omega_T=[0, T)\times \Omega$, ${\partial \over \partial {\bf n}}$
is the outward normal derivative operator. It is called the gradient flow of (\ref{ROFv2}). When $\epsilon=0$, it is called TV flow.
 %Numerical solutions of the above gradient flow have been studied for twenty years.
Similar partial differential equations also appear in geometry
analysis. See references, e.g., \cite{LT78}, \cite{Gerhardt80},
\cite{ABC01}, \cite{ABC01b}, \cite{ABC02}, and the references
therein. The existence, uniqueness, stability of the weak solutions
to these time dependent PDE were studied in the literature mentioned
above. Numerical solution of the PDE (\ref{epsTEROF2}) using finite
elements has been discussed in \cite{FP03} and \cite{FOP05}. In
particular, the researchers showed that the finite element solution
exists, is unique, is convergent to the weak solution of the PDE
(\ref{epsTEROF2}), the rate of convergence under some sufficient
conditions is obtained, and the computation is stable. A fixed point
iterative algorithm for the associated system of nonlinear equations
was discussed in \cite{VO96} and its convergence was studied in
\cite{DV97}. Although the finite difference solution of the time
dependent PDE (\ref{epsTEROF2}) has been the method of choice for
image denoising (e.g. See \cite{VS02}), no convergence of the finite
difference solution to the weak solution of the PDE has been
established in the literature so far to the best of the authors'
knowledge. See also \cite{FL10}.

The purpose of this paper is to provide a proof of the convergence
of the discrete solution obtained from a finite difference scheme
for (\ref{epsTEROF2}) to the weak solution. See our
Theorem~\ref{mainFDconv} in Section 3. Note that the finite
difference scheme in (\ref{FD2}) is slightly different from the
traditional ones: forward or backward or central difference scheme.
We use the average of forward and backward differences. The
advantage of our scheme is that the value of the nonlinear term in
(\ref{ROFv2}) for certain piecewise linear functions is equal to the
value of its discretization of the nonlinear term. As the PDE is
associated with a convex functional, we use the techniques from
convex analysis to help establishing the convergence. In addition,
we study how to numerically solve the time dependent PDE
(\ref{epsTEROF2}) by using our finite difference scheme. As the
finite difference scheme is a system of nonlinear equations, we
shall derive an iterative algorithm and show that the iterative
solutions are convergent. Again we use our techniques on convex
analysis to establish the convergence of the iterative algorithm.
%the PDE (\ref{epsTEROF2}).  we
%shall derive an iterative scheme whose iterative solutions will converge to the solution of the finite difference scheme.

Let us now introduce our finite difference scheme for (\ref{epsTEROF2}). We need some notations.
For convenience, let $\Omega=[0, 1]\times [0, 1]$. We let $N>0$ be a positive integer and divide $\Omega$ by equally-spaced points
$x_i=ih$ and $y_j=jh$ for $0\le i, j \le N-1$ where $h=1/N$. For any $f(x,y)$ defined on $\Omega$, let $f^h_{i,j}=f(x_i,y_j)$ if $f$ is a
continuous function on $\Omega$. Otherwise, $f^h$ will be defined as in (\ref{newdeffh}).
We shall use two different divided differences  $\nabla^+$ and $\nabla^-$ to approximate the gradient operator. That is,
$$
\nabla^+ f^h_{i,j}= \left(\frac{f^h_{i+1,j}-f^h_{i,j}}{h}, \frac{f^h_{i,j+1}-f^h_{i,j}}{h}\right)
$$
and
$$
\nabla^- f^h_{i,j}= \left(\frac{f^h_{i,j}-f^h_{i-1,j}}{h}, \frac{f^h_{i,j}-f^h_{i,j-1}}{h}\right)
$$
for all $0\le i, j\le N-1$ with $f^h_{-1,j} = f^h_{0,j}, f^h_{N,j}=f^h_{N-1,j}$ for all $j$ and
$f^h_{i,-1}=f^h_{i,0}, f^h_{i,N}=f^h_{i,N-1}$ for all $i$. Furthermore, we define discrete divergence operators
$\hbox{div}^+$ and $\hbox{div}^-$ to approximate the continuous divergence operator, i.e.,
\begin{align*}
\hbox{div}^+ (f^h_{i,j}, g^h_{i,j}) =&
\begin{cases}
f^h_{0,j}/h &\qquad i= 0, 0\le j\le N-1 \\
(f^h_{i,j}-f^h_{i-1,j})/h &\qquad 0 < i < N-1, 0\le j\le N-1\\
-f^h_{i-2, j}/h &\qquad i = N-1, 0\le j\le N-1
\end{cases}\\
& \quad +
\begin{cases}
g^h_{i,0}/h &\qquad j= 0, 0\le i\le N-1 \\
(g^h_{i,j}-g^h_{i,j-1})/h &\qquad 0 < j < N-1, 0\le i\le N-1 \\
-g^h_{i, j-2}/h &\qquad j = N-1, 0\le i \le N-1
\end{cases}
\end{align*}
for all $0\le i, j\le N-1$ and similarly for $\hbox{div}^-$. By their definitions, we have for every
$p\in \mathbb{R}^{N\times N} \times \mathbb{R}^{N\times N}$ and $u\in \mathbb{R}^{N\times N}$
$$
\langle-\divp p, u\rangle = \langle p, \nabla^+ u\rangle, \qquad \langle -\divm p, u\rangle = \langle p, \nabla^- u\rangle.
$$

With these notations, we are able to define a finite difference scheme for numerical solution of the time dependent PDE (\ref{epsTEROF2}).
\begin{equation}
\label{FD}
\begin{cases}
\od u_{i,j} = \frac{1}{2}\hbox{ div}^+ \left( \dfrac{\nabla^+ u_{i,j}}{\sqrt{\epsilon + |\nabla^+ u_{i,j}|^2}}
\right)  & \cr
\qquad +\frac{1}{2}\hbox{ div}^- \left( \dfrac{\nabla^- u_{i,j}}{\sqrt{\epsilon +
|\nabla^- u_{i,j}|^2}}
\right)  - \frac{1}{\lambda}(u_{i,j}- f^h_{i,j})  & 0\le i, j\le N-1, t\in [0, T]\cr
{\partial \over \partial {\bf n}} u_{i,j} = 0  & i=0, N, 0\le j\le N-1; \cr
& j=0, N, 0\le i\le N-1,\cr
u(x_i,y_j,0)=u^h_0(x_i,y_j), & 0\le i, j\le N-1,
\end{cases}
\end{equation}
where $u^h_0$ is a discretization of the initial value $u_0$ according to (\ref{newdeffh}).
Next we discretize the time domain $[0, T]$ by equally-spaced points $t_k=k\Dt$, $\Dt = T/M$. We approximate the
$\od u_{i,j}$ by $(u^k_{i,j}-u^{k-1}_{i,j})/\Dt$ to have
the fully discrete version of finite difference scheme:
\begin{equation}
\label{FD2}
\begin{cases}
\frac{1}\Dt(u^k_{i,j}- u^{k-1}_{i,j}) =
\frac{1}{2}\hbox{ div}^+ \left( \dfrac{\nabla^+ u^k_{i,j}}{\sqrt{\epsilon + |\nabla^+ u^k_{i,j}|^2}}
\right)  & \cr
\quad+\frac{1}{2}\hbox{ div}^- \left( \dfrac{\nabla^- u^k_{i,j}}{\sqrt{\epsilon + |\nabla^- u^k_{i,j}|^2}}
\right)  - \frac{1}{\lambda}(u^k_{i,j}- f^h_{i,j})  & 0\le i, j\le N-1, 1\le k\le M\cr
{\partial \over \partial {\bf n}} u^k_{i,j} = 0  & i=0, N, 0\le j\le N-1; \cr
& j=0, N, 0\le i\le N-1,  0\le k \le M\cr
u(x_i,y_j,0)=u^h_0(x_i,y_j), & 0\le i, j\le N-1.
\end{cases}
\end{equation}
% \begin{remark}
% In our numerical scheme, the discrete variation for any array $u^k:=\{u^k_{ij}, 0\le i, j\le N\}$ is in fact defined by
% $$
% |u^k|_{\operatorname{DBV}} = \frac{1}2\sum_{i,j}\sqrt{\epsilon + |\nabla^+u^{k}_{i,j}|^2} +
% \frac{1}2\sum_{i,j}\sqrt{\epsilon + |\nabla^-u^{k}_{i,j}|^2}.
% $$
% This way of defining discrete variation makes it possible to connect discrete and continuous variations by the observation that
% $
% |U|_{\operatorname{BV}} = |u^k|_{\operatorname{DBV}}
% $
% where $U$ is a piecewise linear  function obtained by interpolating $u^k$ over grids on $\Omega$ which will be detailed in Section 3.
%%and further prove the convergence of our discrete solutions to the weak solution.
% The numerical scheme is constructed from the Euler-Lagrange equation of the following variation problem
% $$
% \arg\min E_k(v)
% $$
% for each step $k$ where
% \begin{align*}
% E_k(v) &= \frac{1}{2}\sum_{i,j}\sqrt{\epsilon + |\nabla^+ v_{i,j}|^2}\,h^2 +
% \frac{1}{2}\sum_{i,j}\sqrt{\epsilon + |\nabla^- v_{i,j}|^2}\,h^2 \\
% &\qquad+\frac{1}{2\lambda}\sum_{i,j}(v_{i,j}-f^h_{i,j})^2\,h^2
% +\frac{1}{2\Dt}\sum_{i,j}( v_{i,j} - u^{k-1}_{i,j})^2\,h^2
% \end{align*}
% for all arrays $\{v_{i,j}\}$, $0\le i,j \le N-1$.
% \end{remark}

We shall first show that the above scheme (\ref{FD2}) has a uniqueness solution in \S 2 and we will establish some properties of the solution.
Then we show the solution in (\ref{FD2})
converges to the weak solution of time dependent PDE (\ref{epsTEROF2}) in the sense that the piecewise linear interpolation of the solution
vector of (\ref{FD2}) converges weakly to a function $U^*$ which is the weak solution of the PDE (\ref{epsTEROF2}).
These will be done in \S 3. Next we shall
explain how to numerically solve this system of nonlinear equations in \S 4.  We finally  report our computational results in \S 5.

\section{Preliminary Results}
We first introduce a weak formulation of PDE~(\ref{epsTEROF2}) that is suggested by~\cite{FP03}.
\begin{definition}
\label{weaksol}
We say that  $u \in L^1([0, T],\BV(\Omega))$ is a weak solution of (\ref{epsTEROF2}) if $u$ satisfies the initial value and boundary
conditions in (\ref{epsTEROF2}) and for any
$w \in L^1([0, T], W^{1,1}(\Omega))$ with ${\partial \over \partial {\bf n}}w(x,t) =0$ for all $(t,x)\in [0, T)\times \partial \Omega$,
\begin{equation}
\label{weakform}
\int_0^s\int_\Omega \od u w dx dt  + \int_0^s \int_\Omega \frac{\nabla u\cdot \nabla w}{\sqrt{\epsilon +
|\nabla u|^2}}+ \frac{1}{\lambda}\int_0^s\int_\Omega (u - f)w dxdt =0,
\end{equation}
for any $s\in (0, T]$.
\end{definition}

It is known (cf. \cite{FP03}) there exists a unique weak solution $U^*$ satisfying the above weak formulation.  $U^*$ is in fact in
$L^\infty((0, T],\hbox{BV}(\Omega))$ if $u^0 \in \hbox{BV}(\Omega)$ and $f\in L^2(\Omega)$.
Following the ideas in \cite{LT78}, the researchers in \cite{FP03} further showed the weak solution can be characterized
by the following inequality.
\begin{theorem}
\label{FengThm1}
Let $u$ be a weak solution as in Definition~\ref{weaksol}. Then $u$ satisfies the following inequality:
for any $s\in (0, T]$,
\begin{eqnarray}
\label{FengIne}
&&\int_0^s \int_\Omega \od v (v- u)dxdt + \int_0^s (J(v)- J(u)) dt \cr
&\ge & \frac{1}{2}\left[ \int_\Omega (v(x, s)- u(x, s))^2dx - \int_\Omega (v(x,0)- u_0(x,0))^2dx\right]
\end{eqnarray}
for all $v \in L^1([0, T], W^{1,1}(\Omega))$ with ${\partial \over \partial {\bf n}}v(x,t) =0$ for all $(t,x)\in [0, T)\times
\partial \Omega$, where
\begin{equation}
\label{continuous-J-functional}
J(u) =\int_\Omega \sqrt{\epsilon+|\nabla u(x,t)|^2}dx + \frac{1}{2\lambda} \int_\Omega |f(x,t)- u(x,t)|^2dx.
\end{equation}
On the other hand, if a function $u\in L^1((0, T], \BV(\Omega))$ satisfies the above inequality (\ref{FengIne}),
then $u$ is a weak solution.
\end{theorem}

Theorem~\ref{FengThm1} is our major tool to establish the convergence of the finite difference solution to the weak solution of
the PDE (\ref{epsTEROF2}). We shall use it in the proof of our main result in Theorem~\ref{mainFDconv}.
Next we introduce some basic notations and prove some basic properties
of  the solution vector of finite difference scheme (\ref{FD2}) in the remaining part of this section.

We partition the region $\Omega = [0,1]\times[0,1]$ evenly into $N$
by $N$ grids with a grid size of $h=1/N$, and assume that the pixel value on each grid at index $(i,j)$ is $f^h_{i,j}$,
\begin{equation}
\label{newdeffh}
f^h_{i,j} = \frac{1}{h^2}\int_{ih}^{(i+1)h}\int_{jh}^{(j+1)h} f(x)\,dx,\quad 0\le i,j \le N-1
\end{equation}

Then the initial data $f^h$ for our numerical scheme is a discretization of the initial data $f$ for PDE~(\ref{epsTEROF2}).
\begin{equation}
    \label{projector}
    f^h := \sum_{i,j}f^h_{i,j} \chi_{i,j}(x),
\end{equation}
where $\chi_{i,j}(x)$ is the characteristic function of square $\Omega_{i,j}:=[ih, (i+1)h]\times[jh, (j+1)h]$. When there is no ambiguity, we also
treat array $\{u^k\}$ as a discrete function(piecewise constant on grids) with $u^k(x) = u^k_{i,j}$ for $x \in \Omega_{i,j}$. In later sections, we
will always use superscript(e.g. $u^h(\cdot, t)$ or $u^k$) to indicate that the function is a discrete function. We also introduce a projecting
operator $\Ph$ mapping from $L^1$ to the space of discrete functions
$$
\Ph f := f^h
$$
We define the discrete $L^2$ norms of $f^h$ in analogue of standard $L^2$ norms.
$$
\|f^h\| := \left\{\sum_{i,j}(f^h_{i,j})^2 \,h^2\right\}^{1/2}.
$$
Furthermore, we define a discretized version of the nonlinear  functional (\ref{continuous-J-functional})
\begin{align}
\label{discrete-J-functional}
J^h(v) = \frac{1}{2}\sum_{i,j}\sqrt{\epsilon + |\nabla^+ v_{i,j}|^2}\,h^2 +
\frac{1}{2}\sum_{i,j}\sqrt{\epsilon + |\nabla^- v_{i,j}|^2}\,h^2 +
\frac{1}{2\lambda}\sum_{i,j}(v_{i,j}-f^h_{i,j})^2\,h^2,
\end{align}
and the discrete energy functional
\begin{align}
\label{discrete-E-functional}
E^h(v) = J^h(v) +\frac{1}{2\Dt}\sum_{i,j}( v_{i,j} - u^{k-1}_{i,j})^2\,h^2
\end{align}
for all arrays $v_{i,j}$, $0\le i,j \le N-1$.

We are now ready to show the following existence and uniqueness results.
\begin{theorem}
Fix $N>0$ and $M>0$. There exists a unique array $u^k_{i,j}, 0\le i, j\le N-1, 0\le k\le M$ satisfying the above system (\ref{FD2}) of
nonlinear equations.
\end{theorem}
\begin{proof}
Consider the following minimization problem:
\begin{align}
\label{min-energy}
\min_{v} E^h(v).
\end{align}
The Euler-Lagrange equation for its minimizer $u^k$ is
$$
\partial E^h(u^k) = \partial J^h(u^k) + \frac{u^k - u^{k-1}}{\Delta t}h^2 = 0.
$$
It is straightforward to verify that the subgradient of $J^h$ at $u^k$ is an array with
\begin{eqnarray}
\label{E^h-subgradient}
&\phantom{ {}={}}&\frac{1}{h^2}\partial J^h(u^k)_{i,j} \nonumber\\
&=& %\frac{u^k- u^{k-1}}{\Dt}
- \frac{1}{2}\hbox{ div}^+ \left( \dfrac{\nabla^+ u^k_{i,j}}{\sqrt{\epsilon  + |\nabla^+ u^k_{i,j}|^2}}
\right)
- \frac{1}{2}\hbox{ div}^- \left( \dfrac{\nabla^- u^k_{i,j}}{\sqrt{\epsilon  + |\nabla^- u^k_{i,j}|^2}} \right)
+ \frac{1}{\lambda}(u^k_{i,j}- f^h_{i,j})
%(v_{i,j}- u_{i,j}^k\cr
%&=\sum_{i,j} \partial E^h(u^k_{i,j}) (v_{ij}- u_{i,j}^k)
\end{eqnarray}
Then we have
\begin{align}
\label{euler-lagrange-for-one-step}
\frac{u^k_{i,j}- u^{k-1}_{i,j}}\Dt & -
\frac{1}{2}\hbox{ div}^+ \left( \dfrac{\nabla^+ u^k_{i,j}}{\sqrt{\epsilon + |\nabla^+ u^k_{i,j}|^2}}
\right)
- \frac{1}{2}\hbox{ div}^- \left( \dfrac{\nabla^- u^k_{i,j}}{\sqrt{\epsilon + |\nabla^- u^k_{i,j}|^2}} \right)  \nonumber\\
&\qquad+ \frac{1}{\lambda}(u^k_{i,j}- f^h_{i,j}) = 0,\quad   0\le i, j\le N-1, 1\le k\le M
\end{align}
which is the equation in (\ref{FD2}).
The existence and uniqueness of $u^k_{i,j}$ follows from the strict convexity of the functional $E^h$.
\end{proof}

The following property is a characterization of the discrete solution of~(\ref{FD2}).
\begin{lemma}
\label{LTidea}
Suppose that array $\{u^k_{i,j}, 0\le i,j\le N-1, 0\le k \le M\}$ is a solution of the finite difference
scheme (\ref{FD2}). Then $u^k_{i,j}$ satisfies the following inequality
\begin{eqnarray}
\label{FDineq}
&& \sum_{i,j} \frac{u^k_{i,j}- u^{k-1}_{i,j}}\Dt( v_{i,j}- u^k_{i,j})
 + \frac{1}{2}\left(\sum_{i,j}\sqrt{\epsilon + |\nabla^+ v_{i,j}|^2} -
\sum_{i,j}\sqrt{\epsilon+ |\nabla^+ u^k_{i,j}|^2}\right) + \cr
&& \frac{1}{2}\left(\sum_{i,j}\sqrt{\epsilon + |\nabla^- v_{i,j}|^2}
 - \sum_{i,j}\sqrt{\epsilon+ |\nabla^- u^k_{i,j}|^2}\right)
+\frac{1}{2\lambda}\sum_{i,j}( v_{i,j} - f^h_{i,j})^2 -
\frac{1}{2\lambda}\sum_{i,j}(u^k_{i,j}- f^h_{i,j}) ^2 \cr
&& \ge  0
\end{eqnarray}
for all arrays $v_{i,j}$ that satisfy the Neumann boundary condition.
On the other hand, if an array $\{u^k_{i,j}, 0\le i, j\le N-1, 0\le k \le M\}$ satisfies the above inequality
for all $v_{i,j}$ satisfying the discrete Neumann boundary condition in (\ref{FD2}), then array
$\{u^k_{i,j}, 0\le i, j\le N-1\}$ is a solution of (\ref{FD2}).
\end{lemma}
\begin{proof}
Since $u^k$ is the minimizer of $E^h$, we have the Euler-Lagrange equation
$$
0 = \partial E^h(u^k)
$$
i.e.,
$$
-\frac{u^{k} - u^{k-1}}{\Delta t} h^2 = \partial J^h(u^k).
$$
By the definition of sub-gradient, for any array $v^h_{i,j}$
    %From the Euler-Lagrange equation~(\ref{euler-lagrange-for-one-step}), we have
\begin{eqnarray*}
   - \sum_{i,j} \frac{u^{k}_{i,j} - u^{k-1}_{i,j}}{\Delta t}(v^h_{i,j} - u^k_{i,j}) h^2 \le J^h(v^h) - J^h(u^k).
\end{eqnarray*}
Rearranging terms in the above inequality and the result follows.
\end{proof}

The variation of our scheme is also monotone in the following sense.
\begin{lemma}
\label{variation-monotonicity-lemma}
Define discrete function $u^h(t)$ by
\begin{equation}
    \label{u-t-def}
u^h(t) := \frac{t-t_{k-1}}\Dt u^k + \frac{t_k - t}\Dt u^{k-1}, \qquad t_{k-1}\leq t \leq t_k.
\end{equation}
Then
\begin{align}
\label{variation-monotonicity}
J^h(u^k) \le J^h(u^h(t)), \qquad t_{k-1} \le t \le t_k.
\end{align}
\end{lemma}
\begin{proof}
Since $u^k$ is the minimizer of the following functional
$$
E^h(v) = J^h(v) + \frac{1}{2\Dt}\|u^{k-1}-v\|^2
$$
we have
\begin{equation}
\label{compare-energy}
J^h(u^k) + \frac{1}{2\Dt}\|u^{k-1}-u^k\|^2
\leq J^h(u^h(t)) + \frac{1}{2\Dt}\|u^{k-1}-u^h(t)\|^2.
\end{equation}
For each term in the summation of the $L^2$ square term on the right-hand side,
\begin{align*}
\left|u^{k-1} - u^h(t)\right| &= \left|u^{k-1} - \frac{t-t_{k-1}}\Dt u^k + \frac{t_k - t}\Dt u^{k-1}\right|\\
&= \frac{t-t_{k-1}}\Dt\left|u^k - u^{k-1}\right|\leq \left|u^k - u^{k-1}\right|.
\end{align*}
That is
$$
\frac{1}{2\Dt}\|u^{k-1}-u^h(t)\|^2 \leq \frac{1}{2\Dt}\|u^{k-1}-u^{k-1}\|^2.
$$
With the above inequality, we conclude the result from~(\ref{compare-energy}).
\end{proof}

The following result shows that the computation of finite difference scheme (\ref{FD2}) is stable.
\begin{theorem}
\label{FDstable}
Let $\{u^{k}_f, 0\le k\le M\}$ be the solution of the system of nonlinear equations~(\ref{FD2}) associated with $f^h$
with initial value $u^0_f$. Similarly,
let $\{u_g^k, 0\le k\le M\}$ be the corresponding solution of (\ref{FD2}) associated with $g^h$ with initial value  $u_g^0$.  Then
\begin{equation}
\label{stableinq}
\|u_f^k - u_g^k\| \leq \max\{\|u_f^0 - u_g^0\|, \|f^h - g^h\|\}, \qquad 1 \leq k \leq M.
\end{equation}
\end{theorem}
\begin{proof}
We prove by induction. It is obvious true for $k=0$. Assume the inequality holds for $k-1$. Rearrange the $L^2$ terms in~(\ref{min-energy}).
We have $u_f^k$ is the minimizer of the following problem.
\begin{align}
\label{equiv-min-energy}
\min_{v}  \frac{h^2}{2}\sum_{i,j}\sqrt{\epsilon + |\nabla^+ v_{i,j}|^2} +
\frac{h^2}{2}\sum_{i,j}\sqrt{\epsilon + |\nabla^- v_{i,j}|^2} + (\mu_1 + \mu_2)\left\|v - \left(k_1 f^h + k_2 u_f^{k-1}\right)\right\|^2
\end{align}
where $\mu_1 = 1/(2\lambda), \mu_2 = 1/2\Dt$, and $k_1 =
\mu_1/(\mu_1+\mu_2), k_2 = \mu_2 / (\mu_1+\mu_2)$. By standard
stability property of the minimization problem like
(\ref{equiv-min-energy})(cf. \cite{LW10} or Theorem 3.1 in
\cite{LLM12})
\begin{align*}
\left\|u_f^k - u_g^k\right\| &\leq \left\|\left(k_1 f^h + k_2 u_f^{k-1}\right) - \left(k_1 g^h + k_2 u_g^{k-1}\right)\right\|\\
&\leq k_1\|f^h - g^h\| + k_2 \left\|u_f^{k-1} - u_g^{k-1}\right\| \\
&\leq \max\left\{\|f^h - g^h\|,  \left\|u_f^{k-1} - u_g^{k-1}\right\|\right\}\\
&\leq \max\left\{\|f^h - g^h\|,  \left\|u_f^0 - u_g^0\right\|\right\}.
\end{align*}
This completes the proof.
\end{proof}
\begin{remark}
As a direct deduction, if $g^h = u_g^0 = 0$, the solution $u_g^k$ is also zero for all $k$, then
\begin{align}
\label{l2-bound}
\|u_f^k\| \leq \max\{\|u_f^0\|, \|f^h\|\}, \qquad 1 \leq k \leq M.
\end{align}
\end{remark}

The following lemma discusses the regularity of the discrete solution $u^k$. In image analysis, the input image usually does not have much regularity.
For example, most natural images do not even have weak derivatives.
 Therefore, to model images, we introduce the notation of Lipschitz space, and treat images as functions in this space.
\begin{definition}
\label{def-Lip2alpha}
Let $\alpha\in (0, 1]$ be a real number.
A function $f\in \opLip(\alpha, L^2(\Omega))$  if $f\in L^2(\Omega)$ and the following quantity
\begin{equation}
\label{Lip2alpha}
|f|_{\opLip(\alpha,L^2(\Omega))}:=\sup_{|h|\le 1} \frac{ \| f(\cdot) - f(\cdot+h)\|_{L^2(\Omega_h)}}{|h|^\alpha}
\end{equation}
is finite, where $\Omega_h:=\{x\in \Omega, x+th\in \Omega, \forall t\in [0, 1]\}$. We let
$\|f\|_{\opLip(\alpha,L^2(\Omega))}= \|f\|_{L^2(\Omega)}+|f|_{\opLip(\alpha,L^2(\Omega))}$.
%be the quantity in (\ref{Lipalpha}).
\end{definition}

The parameter $\alpha$ is related to the ``smoothness'' of functions in the Lipschitz space.
Smoother functions belong to Lipschitz spaces with larger $\alpha$ values. For example, a function of bounded variation is a function
in $\opLip(1,L^2(\Omega))$ (cf. \cite{CDPX99}).

\begin{lemma}
\label{l2-smoothness}
Define translation operators $T_{1,0}$ and $T_{0,1}$ by
\begin{eqnarray*}
(T_{1,0}u^k)_{i,j} = u^k_{i+1, j}\qquad 0\le i,j \le N-1\\
(T_{0,1}u^k)_{i,j} = u^k_{i, j+1}\qquad 0\le i,j \le N-1
\end{eqnarray*}
Then if $u_0$ and $f$ in $\hbox{Lip}(\alpha, L^2(\Omega))$,
\begin{align*}
\left\|T_{1,0}u^k - u^k\right\| \leq (\|u^0\|_{\opLip(\alpha, L^2)} + \|f\|_{\opLip(\alpha, L^2)})h^\alpha
\end{align*}
and similarly
$$
\left\|T_{0,1}u^k - u^k\right\| \leq (\|u^0\|_{\opLip(\alpha, L^2)} + \|f\|_{\opLip(\alpha, L^2)})h^\alpha.
$$
\end{lemma}
\begin{proof}
We only prove the first inequality. Recall the Euler-Lagrange equation that
\begin{align*}
\frac{u^{k-1} - u^{k}}{\Delta t} h^2 &= \partial J^h(u^{k}).
\end{align*}
We write the equation element-wisely as
\begin{align*}
\frac{u^k_{i,j}- u^{k-1}_{i,j}}{\Dt} = \frac{1}2\divp\left(\frac{\nabla^+ u^k_{i,j}}{\sqrt{\epsilon + |\nabla^+ u^k_{i,j}|^2}}\right) +
\frac{1}2 \divm\left(\frac{\nabla^- u^k_{i,j}}{\sqrt{\epsilon + |\nabla^- u^k_{i,j}|^2}}\right) - \frac{1}{\lambda}(u^k_{i,j} - f^h_{i,j}).
\end{align*}
Then subtracting the equation at index $(i+1,j)$ from the same equation at index $(i,j)$ for $0 \leq i \le N-2$, we obtain
\begin{align}
\label{difference-at-one-point}
\frac{u^k_{i+1,j}- u^k_{i,j}}{\Dt} - \frac{u^{k-1}_{i+1,j}-u^{k-1}_{i,j}}{\Dt} &= F(\nabla^+ u^k_{i+1,j}, \nabla^+ u^k_{i,j})
\nonumber + F(\nabla^- u^k_{i+1,j}, \nabla^- u^k_{i,j}) \\
&\qquad- \frac{1}\lambda(u^{k}_{i+1,j}-u^{k}_{i,j}) + \frac{1}\lambda(f^h_{i+1,j} - f^h_{i,j})
\end{align}
where $F(\nabla^+ u^k_{i+1,j}, \nabla^+ u^k_{i,j})$ is defined by
$$
F(\nabla^+ u^k_{i+1,j}, \nabla^+ u^k_{i,j}) = \frac{1}2\divp\left(\frac{\nabla^+ u^k_{i+1,j}}{\sqrt{\epsilon +
|\nabla^+ u^k_{i+1,j}|^2}}\right) - \frac{1}2\divp\left(\frac{\nabla^+ u^k_{i,j}}{\sqrt{\epsilon + |\nabla^+ u^k_{i,j}|^2}}\right).
$$
Equation~(\ref{difference-at-one-point}) only holds for $0\le i \le N-2$, $0\le j\le N-1$.
Although equation~(\ref{difference-at-one-point}) is not defined for $i=N-1$,
we can set $u^k_{N+1,j} = u^k_{N,j}$ and $f_{N+1,j} = f_{N,j}$, and equation~(\ref{difference-at-one-point}) still holds.
We multiply~(\ref{difference-at-one-point}) by $u^k_{i+1,j} - u^k_{i,j}$
and add all resulting equations for $0\le i,j \le N-1$ to have
\begin{align*}
&\phantom { {}={} } \frac{1}\Dt\sum_{i,j=0}^{N-1}(u^k_{i+1,j}-u^k_{i,j})^2 \\
&=  \frac{1}\Dt\sum_{i,j=0}^{N-1}(u^{k-1}_{i+1,j}-u^{k-1}_{i,j})(u^k_{i+1,j}-u^k_{i,j}) \\
&\qquad+ \sum_{i,j=0}^{N-1}F(\nabla^+u^k_{i+1,j}, \nabla^+u^k_{i,j})(u^k_{i+1,j} - u^k_{i.j})
+ \sum_{i,j=0}^{N-1}F(\nabla^-u^k_{i+1,j}, \nabla^-u^k_{i,j})(u^k_{i+1,j} - u^k_{i.j}) \\
&\qquad - \sum_{i,j=0}^{N-1}\frac{1}\lambda(u^{k}_{i+1,j}-u^{k}_{i,j})^2 + \sum_{i,j=0}^{N-1}\frac{1}
\lambda(f^h_{i+1,j} - f^h_{i,j})(u^k_{i+1,j}-u^k_{i,j}).
\end{align*}
We show next that the second term is no greater than zero. The third term can be proved to be non-positive similarly. By definition of $F$,
\begin{align*}
&\phantom { {}={} }\sum_{i,j=0}^{N-1}F(\nabla^+ u^k_{i+1,j}, \nabla^+ u^k_{i,j})(u^k_{i+1,j}-u^k_{i,j})\\
&= \sum_{i,j=0}^{N-1}\frac{1}2\divp\left(\frac{\nabla^+ u^k_{i+1,j}}{\sqrt{\epsilon + |\nabla^+ u^k_{i+1,j}|^2}}\right)
(u^k_{i+1,j}-u^k_{i,j}) - \sum_{i,j=0}^{N-1}\frac{1}2\divp\left(\frac{\nabla^+ u^k_{i,j}}{\sqrt{\epsilon + |\nabla^+ u^k_{i,j}|^2}}\right)
(u^k_{i+1,j}-u^k_{i,j}).
\end{align*}
We use the discrete divergence operators and gradient operators to get
\begin{align*}
&\phantom { {}={} }\sum_{i.j=0}^{N-1}F(\nabla^+ u^k_{i+1,j}, \nabla^+ u^k_{i,j})(u^k_{i+1,j}-u^k_{i,j})\\
&= \frac{1}2\sum_{i,j=0}^{N-1}\left(\divp\left(\frac{\nabla^+u^k_{i+1,j}}{\sqrt{\epsilon + |\nabla^+u^k_{i+1,j}|^2}}\right) -
\divp\left(\frac{\nabla^+u^k_{i,j}}{\sqrt{\epsilon + |\nabla^+u^k_{i,j}|^2}}\right)\right)(u^k_{i+1,j}-u^k_{i,j})\\
&= -\frac{1}2\sum_{i,j=0}^{N-1}\left(\left(\frac{\nabla^+u^k_{i+1,j}}{\sqrt{\epsilon + |\nabla^+u^k_{i+1,j}|^2}}\right) -
\left(\frac{\nabla^+u^k_{i,j}}{\sqrt{\epsilon + |\nabla^+u^k_{i,j}|^2}}\right)\right)(\nabla^+u^k_{i+1,j}-\nabla^+u^k_{i,j})\\
&\qquad - \sum_{j=0}^{N-1}\frac{|\nabla^+ u^k_{0,j}|^2}{\sqrt{\epsilon + |\nabla^+ u^k_{0,j}|^2}}
\end{align*}
Each term in the first sum is non-negative due to the following inequality: for any $x, y\in {\bf R}^2$,
$$
\left(\frac{x}{\sqrt{\epsilon + |x|^2}} - \frac{y}{\sqrt{\epsilon+|y|^2}}\right)(x-y) \ge 0
$$
which can be verified easily. 
By similar arguments, one has
$$
\sum_{i,j=0}^{N-1}F(\nabla^-u^k_{i+1,j},\nabla^-u^k_{i,j})(u^k_{i+1,j}-u^k_{i,j}) \leq 0
$$
It follows
\begin{align*}
\frac{1}\Dt\sum_{i,j=0}^{N-1}(u^k_{i+1,j}-u^k_{i,j})^2
&\le  \frac{1}\Dt\sum_{i,j=0}^{N-1}(u^{k-1}_{i+1,j}-u^{k-1}_{i,j})(u^k_{i+1,j}-u^k_{i,j}) \\
&\qquad - \sum_{i,j=0}^{N-1}\frac{1}\lambda(u^{k}_{i+1,j}-u^{k}_{i,j})^2 + \sum_{i,j=0}^{N-1}\frac{1}
\lambda(f^h_{i+1,j} - f^h_{i,j})(u^{k}_{i+1,j}-u^{k}_{i,j}).
\end{align*}
We rewrite the sums in form of discrete integrals and discrete inner products, and apply the arithmetic-geometric inequality
\begin{align*}
\frac{1}\Dt \| T_{1,0}u^k - u^k\|^2
&\leq \frac{1}\Dt \left\langle T_{1,0}u^{k-1} - u^{k-1}, T_{1,0}u^k - u^k\right\rangle \\
&\qquad- \frac{1}\lambda \|T_{1,0}u^k - u^k\|^2 + \frac{1}\lambda\left\langle T_{1,0}f- f, T_{1,0}u^k - u^k\right\rangle \\
&\leq \frac{1}{2\Dt}\|T_{1,0}u^{k-1} - u^{k-1}\|^2 + \frac{1}{2\Dt}\|T_{1,0}u^k - u^k\|^2\\
&\qquad - \frac{1}{2\lambda}\|T_{1,0}u^k - u^k\|^2 + \frac{1}{2\lambda}\|T_{1,0}f - f\|^2.
\end{align*}
Rearrange and combine similar terms to have
\begin{equation}
\label{recursive-inequality}
(\frac{1}\Dt + \frac{1}\lambda)\|T_{1,0}u^k - u^k\|^2 \leq \frac{1}\Dt\|T_{1,0}u^{k-1} - u^{k-1}\|^2 + \frac{1}\lambda\|T_{1,0}f - f\|^2.
\end{equation}
We now prove the following inequality by induction
\begin{equation}
\label{real-result}
\|T_{1,0}u^k - u^k\|^2 \leq \max\{\|T_{1,0}u^0 - u^0\|^2, \|T_{1,0}f - f\|^2\}.
\end{equation}
It is obvious true for $k=0$. Assuming the inequality holds for $k-1$, one can easily see that it also holds for
$k$ by~(\ref{recursive-inequality}). Therefore, one has
\begin{align*}
\|T_{1,0}u^k - u^k\| &\leq \|T_{1,0}u^0 - u^0\| + \|T_{1,0}f - f\|
\leq(\|u^0\|_{\opLip(\alpha, L^2)} + \|f\|_{\opLip(\alpha, L^2)})h^{\alpha}.
\end{align*}
This completes the proof.
\end{proof}

\section{Main Result and Its Proof}
In this section, we shall show that the piecewise linear interpolation of the
solution vector of the finite difference scheme (\ref{FD2}) converges weakly to the
solution of the gradient flow~(\ref{epsTEROF2}).
We assume  that the array $\{u^k_{i,j}, 0\le i, j\le N-1, 0\le k\le M\}$ is the solution vector of (\ref{FD2}).

To connect the discrete solution $\{u^k_{i,j}\}$ of~(\ref{FD2}) and the ``continuous'' weak solution of~(\ref{epsTEROF2}), we first construct a
function $U_{N,M}(\cdot, t)$ in $W^{1,1}(\Omega)$ for each $t \in [0, T]$ in the form of a linear interpolation of $u^k$.

Let $\Delta_N$ be the following type of triangulation of $\Omega = [0,1]\times[0,1]$ with vertices $((i+1/2)h, (j+1/2)h), 0 \le i,j \le N-1$, $h=1/N$.
Suppose the base functions of the continuous linear finite element space $S^0_1(\Delta_N)$ are $\{\phi_{i,j}(x), (i,j) \in \mathbb{Z}^2\}$, where
 $\phi_{i,j}$ is a scaled and translated standard continuous linear box spline function $\phi(x)$ based on three directions $e_1=(1,0), e_2(0,1)$
and $e_3=(-1,1)$, i.e.  $\phi_{i,j}(x) := \phi(x/h - (i+1/2, j+1/2))$ for any $(i,j) \in \mathbb{Z}^2$.
%\begin{eqnarray*}
%\phi(i,j)&=&\begin{cases} 1,&(i,j)=(0,0),\\ 0,&(i,j)\neq (0,0).\end{cases} \quad (i,j)\in \mathbb Z^2\\
%\phi_{i,j}(x) &=& \phi(x/h - (i+1/2, j+1/2)).
%\end{eqnarray*}

\begin{center}
\begin{picture}(100,140)(0,-20)
\put(0,0){\line(1,0){100}}
\put(0,0){\line(0,1){100}}
\put(0,20){\line(1,0){100}}
\put(0,40){\line(1,0){100}}
\put(0,60){\line(1,0){100}}
\put(0,80){\line(1,0){100}}
\put(0,100){\line(1,0){100}}
\put(20,0){\line(0,1){100}}
\put(40,0){\line(0,1){100}}
\put(60,0){\line(0,1){100}}
\put(80,0){\line(0,1){100}}
\put(100,0){\line(0,1){100}}
\put(0,20){\line(1,-1){20}}
\put(0,40){\line(1,-1){40}}
\put(0,60){\line(1,-1){60}}
\put(0,80){\line(1,-1){80}}
\put(0,100){\line(1,-1){100}}
\put(20,100){\line(1,-1){80}}
\put(40,100){\line(1,-1){60}}
\put(60,100){\line(1,-1){40}}
\put(80,100){\line(1,-1){20}}
\put(-10,-10){Fig. 1. A triangulation}
\end{picture}
\end{center}

For any $k$, we define piecewise linear interpolation $U_{N,M}(x, t_k)$ of $u^k$ on $\Omega$ by
\begin{align}
\label{U_definition}
U_{N,M}(x, t_k) := \sum_{i,j=0}^{N-1}u^k_{i,j}\phi_{i,j}(x).
\end{align}
Having defined $U_{N,M}(\cdot,t_k)$ for $k = 0,\cdots,M$ on $\Omega$, we further define $U_{N,M}(\cdot, t)$ for $ t_{k-1} \leq t \leq t_k$
by linear interpolating $U_{N,M}(\cdot,t_{k-1})$ and $U_{N,M}(\cdot,t_k)$ on interval $[t_{k-1}, t_k]$.
$$
U_{N,M}(\cdot,t) = \frac{t-t_{k-1}}\Dt U_{N,M}(\cdot, t_k) + \frac{t_k-t}\Dt U_{N,M}(\cdot, t_{k-1}).
$$
By the definition of $u^h(t)$ given in~(\ref{u-t-def}), we can also write $U_{N,M}(\cdot, t)$ as
$$
U_{N,M}(\cdot, t) = \sum_{i,j=0}^{N-1} u^h(t)\phi_{i,j}
$$
We next prove a sequence of lemmas to explain the properties of $U_{N,M}(\cdot, t)$.

\begin{lemma}
\label{FDTbound}
Suppose $u_0\in W^{1,1}(\Omega), f\in L^2(\Omega)$. For any $t \in [0, T]$,
$\| \od U_{N,M}(\cdot, t)\|_{L^2(\Omega_T)} <C$ for a positive constant $C$ only depending on $u_0$ and $f$.
\end{lemma}
\begin{proof}
Let us write the Euler-Lagrange equation~(\ref{euler-lagrange-for-one-step}) in a concise format:
\begin{align*}
\frac{u^{k-1} - u^k}{\Dt}h^2 &= \partial J^h(u^k).
\end{align*}
The equation above holds element-wise at each index $(i,j)$. For the equation at each index $(i,j)$, we
multiply both sides by $u^{k-1}_{i,j}-u^k_{i,j}$ and then add the equations for all $(i,j)$. In terms of
the standard inner product notation, we write the result in the following form:
\begin{align*}
\left\langle \frac{u^{k-1} - u^k}\Dt, u^{k-1} - u^k\right\rangle = \left\langle \partial J^h(u^k), u^{k-1} - u^k \right\rangle
\end{align*}
By the definition of sub-differential $\partial J^h(u^k)$
\begin{align*}
\left\langle \frac{u^{k-1} - u^k}\Dt, u^{k-1} - u^k\right\rangle
&= \left\langle \partial J^h(u^k), u^{k-1} - u^k \right\rangle
\leq J^h(u^{k-1})   - J^h(u^k).
\end{align*}
We have
\begin{align*}
\frac{1}\Dt\|u^{k-1} - u^k\|^2 \leq J^h(u^{k-1}) - J^h(u^k), \qquad 1 \leq k \leq M.
\end{align*}
Add the above inequalities for $k = 1,\cdots, M$,
\begin{equation}
\label{l2-bound-discrete-derivative}
\sum_{k=1}^M \frac{1}\Dt\|u^{k-1} - u^k\|^2 \leq J^h(u^0) - J^h(u^M).
\end{equation}
Note that
\begin{equation*}
    \frac{dU_{N,M}(\cdot, t)}{dt} = {\sum_{i,j}}\frac{u^k_{i,j} - u^{k-1}_{i,j}}\Dt {\phi_{i,j}}, \qquad t^{k-1}<t<t_k.
\end{equation*}
Then applying Cauchy-Schwarz inequality with $|\phi_{i,j}(x)|\le 1$, we have
\begin{align*}
    \left\|\frac{dU_{N,M}}{dt}\right\|^2_{L^2(\Omega_T)}  &= \sum_{k=1}^M \int_\Omega \left|\sum_{i,j}\frac{u^k_{i,j} - u^{k-1}_{i,j}}\Dt {\phi_{i,j}}
\right|^2\,dx \Delta t\\
    &\le 9\sum_{k=1}^M \left\|\frac{u^k - u^{k-1}}{\Delta t}\right\|^2 \Delta t \le 9(J^h(u^0) - J^h(u^M)).
\end{align*}
where $u^0= \Ph u_0$.  Here $9$ above can be replaced by 1 using Lemma 2.4 in \cite{LLM12}. 
Note that $J^h(u^0)$ is bounded by a positive constant independent of $h$
when $u_0\in W^{1,1}(\Omega)$.  This completes the proof.
\end{proof}

\begin{lemma}
\label{FDL2bound}
Suppose $u^0, f\in L^2(\Omega)$. Then $\|U_{N,M}\|_{L^2(\Omega_T)} \le C$ for a constant $C$ only dependent on $f$ and $u^0$. Furthermore,
$\|U_{N,M}(\cdot, t)\|_{L^2(\Omega)} \le C$ for a positive constant $C$ for any $t\in [0, T]$.
\end{lemma}
\begin{proof}
We use~(\ref{l2-bound}) to bound $\|U_{N,M}\|_{L^2(\Omega_T)}$ and
$\|U_{N, M}(\cdot, t)\|_{L^2(\Omega)}$. Recall $u^0_f=u^0$. It is
easy to see  for $t = t_k$,
\begin{align*}
\|U_{N,M}(\cdot, t_k)\|^2_{L^2(\Omega)} \le \|u^k_f\|^2 \le \max\{\|u^0_f\|, \|f^h\|\}^2.
\end{align*}
(cf. \cite{LW10} or Lemma 2.4 in \cite{LLM12} for the first
inequality and Remark 2.1 or (\ref{l2-bound}) for the second
inequality). Then we have
\begin{align*}
\|U_{N,M}\|^2_{L^2(\Omega_T)} &=
\int_0^T\left\|U_{N, M}(\cdot, t)\right\|^2_{L^2(\Omega)}\,dt\\
&= \sum_{k=1}^M\int_{t_{k-1}}^{t_k} \left\|\frac{(t-t_{k-1}) U_{N,M}(\cdot, t_k) + (t_k - t) U_{N,M}(\cdot, t_{k-1})}
{\Delta t}\right\|^2_{L^2(\Omega)}\,dt\\
&\leq \sum_{k=1}^M \int_{t_{k-1}}^{t_k}\|U_{N,M}(\cdot, t_k)\|^2_{L^2(\Omega)}+\|U_{N,M}(\cdot, t_{k-1})\|^2_{L^2(\Omega)}\,dt\\
&\leq \sum_{k=1}^M \int_{t_{k-1}}^{t_k}\|u^k\|^2+\|u^{k-1}\|^2\,dt \leq 2TC^2.
\end{align*}
As discussed above, for each $t\in [0, T]$, the integrand is $\|U_{N, M}(\cdot, t)\|^2_{L^2(\Omega)}$ which is less than or equal to $2C^2$ by
(\ref{l2-bound}). These complete the proof.
\end{proof}

The above two lemmas ensure that there exists a convergent subsequence from $\{U_{N, M}, N, M\to \infty\}$ and a function $U^*\in L^2(0,T,L^2(\Omega))$
such that $U_{N,M}$ and $\od U_{N,M}$ weakly converge to $U^*$ and $\od U^*$ in $L^2(\Omega_T)$.

%Now we use the piecewise constant injection $I_Nf$ of $f$ to replace the $f$ in the above and
%$I_N u_0$ to replace $u_0$. Also,

Recall the definition of $u^h(t)$ in~(\ref{u-t-def}) with $u^k=(u^k_{ij}, 0\le i, j\le N-1)$.
That is, $u^h(\cdot, t)$ is a piecewise linear function in $t$ while piecewise constant function in $x$. However, $U_{N,M}$ is a piecewise linear
function in $x\in \Omega$ and piecewise linear function in $t$.
We now further show
\begin{lemma}
\label{useful}
Suppose $f, u_0\in \opLip(\alpha, L^2(\Omega))$.  Then
$$
\|U_{N,M}(\cdot, t)- u^h(\cdot, t)\|_{L^1([0,T];L^2(\Omega_T))} \le C T (\|u^0\|_{\opLip(\alpha, L^2)} + \|f\|_{\opLip(\alpha, L^2)})h^\alpha
$$
for a positive constant $C$ dependent only on $ f$ and $u_0$.
\end{lemma}
\begin{proof}
Let $g(x, t) = U_{N,M}(x, t)- u^h(x, t)$. For any $x$, $g(x, t)$ is a linear function of $t$. A direct calculation shows
\begin{equation*}
\int^{t_k}_{t_{k-1}}\|g(x, t)\|_{L^2(\Omega)}\,dt \le \frac{1}{2}\left(\|g(x, t_k)\|_{L^2(\Omega)}
+ \|g(x, t_{k-1})\|_{L^2(\Omega)}\right)(t_k - t_{k-1}).
\end{equation*}
Adding these inequalities for $k = 1, \cdots, M$, we have
\begin{align}
\label{basic-bound}
\int_0^T \|g(x, t)\|_{L^2(\Omega)}\,dt \leq \Delta t\sum_{k=0}^M\|g(x, t_k)\|_{L^2(\Omega)}.
\end{align}
Then we only need to bound $\|g(x, t_k)\|$. We note that $g(x, t)$ is a piecewise linear function of $x$ on each sub-grid
$\Omega_{i,j}:=[ih,(i+1)h]\times[jh, (j+1)h]$, $0 \le i,j \le N-1$ for any $t$. Tedious calculation gives
\begin{align*}
\|g(x, t_k)\|^2_{L^2(\Omega)} &= \sum_{i,j} \int_{\Omega_{i,j}}|U_{N,M}(x, t_k) - u^h(x, t_k)|^2 \\
&\leq \sum_{i,j} Ch^2 \left(\left|u^k_{i+1,j} - u^k_{i,j}\right|^2 + \left|u^k_{i,j+1} - u^k_{i,j}\right|^2+\left|u^k_{i-1,j}
- u^k_{i,j}\right|^2 + \left|u^k_{i,j-1} - u^k_{i,j}\right|^2\right)\\
&\leq C\left( \left\|T_{1,0}u^k - u^k\right\|^2 + \left\|T_{0,1}u^k - u^k\right\|^2\right)\\
&\leq 2C(\|f\|_{\opLip(\alpha,L^2)} + \|u_0\|_{\opLip(\alpha,L^2)})^2 h^{2\alpha}.
\end{align*}
The last line follows from Lemma~\ref{l2-smoothness}. We substitute the bound for the $\|g(x, t_k)\|_{L^2(\Omega)}$ in
inequality~(\ref{basic-bound}) to complete the proof.
\end{proof}

%Recall that for each $v_N \in S^0_1(\Delta_N)$, we write
%$$
%v_N = \sum_{i, j}v_{i, j}\phi_{i, j},
%$$
%where $\phi_{i,j}$ is a basis of $S^0_1(\Delta_N)$.
\begin{lemma}
\label{spline-approximation}
For all functions $v$ in $L^1([0,T], W^{1,1}(\Omega))$, there is a sequence of functions $\{v_N\}$ in $L^1([0,T], S^0_1(\Delta_N))$ so that
\begin{equation}
\label{spline-approximation1}
\lim_{N \to \infty} \|v - v_N\|_{L^1([0,T]; L^2(\Omega))} = 0.
\end{equation}
and
\begin{equation}
\label{spline-approximation2}
\lim_{N \to \infty}\left\|v - v_N\right\|_{L^1([0,T]; W^{1,1}(\Omega))} = 0
\end{equation}
\end{lemma}
\begin{proof}
For any $0 \le t \le T$, define the interpolant $\mathcal{I}^h v$ for $v(\cdot,t)$ in $C(\Omega)$ by
\begin{align*}
\mathcal{I}^hv(x, t) = \sum_{i,j} v( (i+1/2)h, (j+1/2)h,t )\phi_{i,j}(x).
\end{align*}
And for any $t \in [0, T]$, define
\begin{align}
    \label{v_N-def}
    v_N(x, t) = \mathcal{I}^hv_\epsilon(x, t)
\end{align}
where $v_\epsilon$ is the smoothed $v$  by a symmetric smooth cut-off function $\psi_\epsilon$ satisfying
(i) $\mbox{supp} \psi_\epsilon \subset B(0, \epsilon)$ and (ii) $\int_{\mathbb{R}^2} \psi_\epsilon \,dx = 1$. More precisely,
$$
v_\epsilon = \int_{\mathbb{R}^2} v(x-y)\psi_\epsilon(y)\,dy.
$$
Since we need to use the value of $v$ outside $\Omega$ in the above integration,
we extend $v$ to all of $\mathbb R^2$ by reflecting and translating; Define
\begin{equation*}
v(x_1,x_2, t)=v(2-x_1, x_2, t),\qquad \mbox{for }1\leq x_1\leq 2,\ 0\leq x_2\leq 1,
\end{equation*}
and
\begin{equation*}
v(x_1,x_2, t)= v(x_1, 2-x_2, t), \qquad \mbox{for }0\leq x_1\leq 2,\ 1\leq x_2\leq 2.
\end{equation*}
Having extended $v$ on $2\Omega$, we then extend $v$ periodically on all of $\mathbb R^2$.

It is a classical result(cf. \cite{Ziemer89}) that for $0 \le t \le
 T$,
\begin{equation}
    \label{v_epsilon-w11-bound}
|v_\epsilon(\cdot, t)|_{W^{1,1}(\Omega)} \le |v(\cdot, t)|_{W^{1,1}(\Omega)},
\end{equation}
and
\begin{align}
    \label{ineq-part1}
\lim_{\epsilon \to 0} \|v_\epsilon(\cdot, t) - v(\cdot, t)\|_{W^{1,1}(\Omega)} = 0.
\end{align}

We also know $\mathcal{I}^h$ is a bounded operator from
$C^2(\overline{\Omega})$ to $W^{1,1}(\Omega)$, and(cf. \cite{BS94}
or  \cite{LW10})
\begin{align}
    \label{int-boundness}
    |v_\epsilon(\cdot, t) - \CIh v_\epsilon(\cdot, t)|_{W^{1,1}(\Omega)} \le Ch|v_\epsilon(\cdot, t)|_{W^{2,1}(\Omega)}\le C\frac{h}\epsilon |v(\cdot, t)|_{W^{1,1}(\Omega)}
\end{align}
\begin{align}
    \label{l1-boundness}
    \|v_\epsilon(\cdot, t) - \mathcal{I}^hv_\epsilon(\cdot, t)\|_{L^1(\Omega)} \le Ch|v_\epsilon(\cdot, t) - v(\cdot, t)|_{W^{1,1}(\Omega)} \le 2Ch|v(\cdot, t)|_{W^{1,1}(\Omega)}.
\end{align}
Setting $\epsilon = h^{1-\alpha}$, we have
\begin{equation}
    \label{v_epsilon-int-v-bound}
    \|v_\epsilon(\cdot, t) - \mathcal{I}^h v_\epsilon(\cdot, t)\|_{W^{1,1}(\Omega)} \le Ch^{\alpha}|v(\cdot, t)|_{W^{1,1}(\Omega)},
\end{equation}
and
\begin{equation}
    \label{ineq-part2}
\lim_{h \to 0}\left\|v_\epsilon(\cdot, t) - \mathcal{I}^hv_\epsilon(\cdot, t)\right\|_{W^{1,1}(\Omega)} = 0.
\end{equation}
Finally inequality~(\ref{spline-approximation2}) follows
from~(\ref{ineq-part1}),~(\ref{ineq-part2}) and Legesuge's Dominated
Convergence Theorem. Inequality~(\ref{spline-approximation1})
follows from Sobolev embedding theorem(cf. \cite{Ziemer89}, Remark
2.5.2)
\begin{equation}
\label{problematic}
\left\|v(\cdot, t) - \mathcal{I}^hv_\epsilon(\cdot, t)\right\|_{L^2(\Omega)} \le C \left\|v(\cdot, t) -
\mathcal{I}^hv_\epsilon(\cdot, t)\right\|_{W^{1,1}(\Omega)}
\end{equation}
and equation~(\ref{spline-approximation2}).
\end{proof}

We now bound the difference between the two projecting operators: $\mathcal{I}^h v_\epsilon$ and $\Ph v_\epsilon$
\begin{lemma}
\label{approximation}
    For any $v \in W^{1,1}(\Omega)$,
    \begin{equation}
        \|\CIh v_\epsilon - \Ph v_\epsilon\| \le Ch|v|_{W^{1,1}(\Omega)}.
    \end{equation}
\end{lemma}
\begin{proof}
    \begin{equation*}
        \|\CIh v_\epsilon - \Ph v_\epsilon\| \le \|\CIh v_\epsilon - v_\epsilon\| + \|v_\epsilon - \Ph v_\epsilon\|.
    \end{equation*}
Now the result follows from~(\ref{l1-boundness}) and
Poincar\'{e}-Wirtinger inequality(cf.~\cite{av})
\begin{equation*}
    %\label{p-inequality}
    \|v_\epsilon - \Ph v_\epsilon\|_{L^2(\Omega)} \le Ch|v_\epsilon|_{\BV(\Omega)}.
\end{equation*}
\end{proof}

We have introduced two notations of total variation, one for functions in $\BV(\Omega)$ and the other one for discrete functions.
We need to show these two versions of total variation are consistent.
%For the piecewise linear spline function $U_{N,M}(\cdot, t_k)$ defined in~(\ref{U_definition}), although its total variation equals the variation term in $J^h(u^k)$, the $L^2$ term in $J(U_{N,M})$ does not equal the discrete $\ell^2$ sum in $J^h(u^k)$. Similarly, for $v_N(\cdot, t) = \sum_{i,j} v^h_{\epsilon, i, j}(t)\phi_{i,j}$, $J(v_N)$ does not equal to $J^h(v^h_\epsilon)$ where $v^k_\epsilon$ is the array $\{v^h_{\epsilon, i, j}\}$.
We use the following lemma to bound the difference between the continuous variation $J(U_{N,M}(\cdot, t))$
and the discrete variations $J(u^k)$. We bound the difference between  $J(v_N(\cdot, t))$ and $J(v^h_\epsilon)$ similarly.

\begin{lemma}
\label{J-difference}
Let $\{v_N\}$ be the sequence of functions  defined as in Lemma~\ref{spline-approximation}. Then for any $t \in [0, T]$
\begin{align}
\label{J-difference1}
|J(v_N(\cdot, t)) - J^h(v^h_\epsilon(t))| \le Ch^\alpha,
\end{align}
where $C$ depends on $v$ and $f$. Moreover, for $U_{N,M}(\cdot, t)$ defined in~(\ref{U_definition}) we have
\begin{align}
\label{J-difference2}
|J(U_{N,M}(\cdot, t)) - J^h(u^h(t))| \le Ch^{\alpha},
\end{align}
where $C$ depends on $f$.
\end{lemma}
\begin{proof}
%comment-lai
Note that for any function $v_N(\cdot, t)$ in $S^0_1(\Delta_N)$, the variation term in $J(v_N(\cdot, t))$ is exactly equal to the variation term in
$J^h(v^h_\epsilon(t))$. This is why we design our finite difference schemes in (\ref{FD}) and (\ref{FD2}) instead of the standard forward difference or
backward difference scheme. 
We only need to bound the difference between the second terms in $J(v_N)$ and $J^h(v^h_\epsilon)$.

Let $v^h_{\epsilon, i, j}(t)$ be the value of $v_\epsilon(\cdot, t)$ at point $ ((i+1/2)h, (j+1/2)h)$.
%Let $\Qh v_\epsilon(\cdot,t)$ be a mapping from $C(\Omega)$ to the space of piecewise constant functions
Define discrete function $v^h_\epsilon(t)$ by
\begin{equation}
    \label{Q-operator}
    v^h_\epsilon(x, t) := \sum_{i,j} v^h_{\epsilon,i,j}(t)\chi_{i,j}(x),
\end{equation}
and recall  $f^h$ is the piecewise constant projection of $f$, i.e. $f^h = \Ph f$.
\begin{align*}
    &\phantom{ {}={} } |J(v_N(\cdot,t)) - J^h(v^h_\epsilon(t))|
    = \left|\frac{1}{2\lambda}\| v^h_\epsilon(\cdot,t) - f^h\|^2 - \frac{1}{2\lambda}\|\mathcal{I}^h v_\epsilon(\cdot,t) - f\|^2\right|\\
    & = \frac{1}{2\lambda}\bigg|(\| v^h_\epsilon(\cdot,t) -  f^h\| - \|\mathcal{I}^h v_\epsilon(\cdot,t) - f\|)(\| v^h_
\epsilon(\cdot,t)-f^h\|+\|\mathcal{I}^h v_\epsilon(\cdot,t) - f\|)\bigg|\\
    &\le \frac{1}{2\lambda}\left(\| v^h_\epsilon(\cdot,t) - \mathcal{I}^h v_\epsilon(\cdot,t)\| + \| f^h - f\|\right) C(\|v_\epsilon(\cdot,t)\|+\|f\|)
\end{align*}
By standard approximation theory(cf.~\cite{Ziemer89}) and Sobolev
inequality
\begin{equation*}
    \|v^h_\epsilon - \mathcal{I}^h v_\epsilon\| \le Ch\|D v_\epsilon\| \le Ch(|v_\epsilon|_{W^{1,1}} + |v_\epsilon|_{W^{2,1}}) \le  C\frac{h}\epsilon
|v|_{W^{1,1}},
\end{equation*}
and
\begin{equation*}
    \|f^h - f\| \le C |f|_{\opLip(\alpha, L^2)}h^\alpha.
\end{equation*}
Then we proved inequality~(\ref{J-difference1}) by setting $\epsilon
= h^{1-\alpha}$. We can prove~(\ref{J-difference2}) along the same
line of arguments(noting $\|u^h\| \le 2\|f\|$ and applying
Lemma~\ref{l2-smoothness}. We omit the details.
\end{proof}

The following proposition is another one of the key ingredients to prove our main results in Theorem~\ref{mainFDconv}.
\begin{proposition}
\label{proposition1}
For any test functions $v$ in $L^1([0,T], W^{1,1}(\Omega))$, let $\{ v_N \}$ be a sequence defined in Lemma~\ref{spline-approximation}. t
Then for $0<s<T$
\begin{equation}
\label{keyineq21}
\int_0^s \left[\int_\Omega \od U_{N,M} (v_N- U_{N,M})dx + (J(v_N)- J(U_{N,M}))\right] dt
\ge  -\hbox{Err}_{N,M}
\end{equation}
where $\hbox{Err}_{N,M}$ depends on $v$ and tends to zero as $N,M \to \infty$ in the following fashion
\begin{equation}
\label{laimj}
\frac{h^\alpha}{\Delta t}  = \frac{M}{T N^\alpha} \to 0.
\end{equation}
\end{proposition}
\begin{proof}
The idea of the proof is to rewrite the left-hand side of~(\ref{keyineq21}) as the left-hand side of (\ref{FDineq}) plus some error and bound the
error. As the preparation for a long calculation, we first remind the reader that for $t\in (t_{k-1}, t_k)$,
$$
U_{N,M}(\cdot,t)=U_{N,M}(\cdot, t_{k-1})(t_k-t)/\Dt+U_{N,M}(\cdot,t_k)(t-t_{k-1})/\Dt
$$
and
\begin{equation}
\label{Flux}
\od U_{N,M}(\cdot, t)=\frac{U_{N,M}(\cdot, t_k)- U_{N,M}(\cdot,t_{k-1})}{\Dt}.
\end{equation}
and
$
v_N(\cdot, t) = \mathcal{I}^h v_\epsilon(\cdot, t)
$
as defined in~(\ref{v_N-def}).

Without loss of generality, we consider the integration over $[0, T]$ instead of $[0, s]$. We rewrite the first term of the left-hand side of
~(\ref{keyineq21}) as
\begin{align}
\label{temp-1}
&\phantom{ {}={}} \int_0^T\int_\Omega \od U_{N,M} (v_N(\cdot,t)-U_{N,M}(\cdot.t))\,dxdt \nonumber \\
&= \sum_{k=1}^{M}\int_{t_{k-1}}^{t_k} \int_\Omega \frac{U_{N,M}(\cdot, t_k)- U_{N,M}(\cdot,t_{k-1})}{\Dt}(v_N(\cdot, t)
- U_{N,M}(\cdot, t))\,dxdt\nonumber\\
&= \sum_{k=1}^{M}\int_{t_{k-1}}^{t_k} \int_\Omega \frac{U_{N,M}(\cdot, t_k)
- U_{N,M}(\cdot,t_{k-1})}{\Dt}( v_N(\cdot, t) - U_{N,M}(\cdot, t_{k}))\,dxdt + \hbox{Err}_1.
\end{align}
where
\begin{align*}
\hbox{Err}_1 &= \sum_{k=1}^M\int_{t_{k-1}}^{t_k}\int_\Omega \frac{U_{N,M}(\cdot, t_k)- U_{N,M}(\cdot,t_{k-1})}{\Dt}( U_{N,M}(\cdot,
t_{k}) - U_{N,M}(\cdot, t))\,dxdt\\
& = \sum_{k=1}^M\int_{t_{k-1}}^{t_k}\int_\Omega \frac{U_{N,M}(\cdot, t_k)- U_{N,M}(\cdot,t_{k-1})}{\Dt}( U_{N,M}(\cdot, t_{k}) - U_{N,M}(\cdot,
t_{k-1}))\frac{t_k-t}{\Dt}\,dxdt.\\
\end{align*}
We bound $\hbox{Err}_1$ by
\begin{align*}
\left|\hbox{Err}_1\right|  &\leq
\sum_{k=1}^M\int_{t_{k-1}}^{t_k}\int_\Omega \left|\frac{U_{N,M}(\cdot, t_k)- U_{N,M}(\cdot,t_{k-1})}{\Dt}( U_{N,M}(\cdot, t_{k}) - U_{N,M}
(\cdot, t_{k-1}))\frac{t_k-t}{\Dt}\right|\,dxdt.\\
&\leq \sum_{k=1}^M \int_\Omega \left|\frac{U_{N,M}(\cdot, t_k)- U_{N,M}(\cdot,t_{k-1})}{\Dt}( U_{N,M}(\cdot, t_{k}) - U_{N,M}
(\cdot, t_{k-1}))\,dx\right|\Dt \\
&= \Dt \left\|\frac{dU_{N,M}}{dt}\right\|^2_{L^2(\Omega_T)} \le C\Dt,
\end{align*}
where  the last inequality comes from Lemma~\ref{FDTbound}.

To apply the characteristic inequality~(\ref{FDineq}), we need to replace all the piecewise linear functions in~(\ref{temp-1})
by piecewise constant functions and bound the introduced error. Recall discrete functions $v^h_\epsilon(\cdot, t)$ and $u^h(\cdot, t)$ defined in~(\ref
{Q-operator}) and~(\ref{u-t-def}) respectively. We replace $v_N(\cdot, t), U_{N,M}(\cdot, t)$ in~(\ref{temp-1}) by $v^h_\epsilon(\cdot, t)$, and
$u^h(\cdot, t)$ respectively and add an error term. To simplify the presentation, we introduce the following notations to
denote the difference between a continuous function and a piecewise constant function;
\begin{eqnarray*}
    \Delta v_N(\cdot, t) &:=& v_N(\cdot, t) - v^h_\epsilon(\cdot, t), \\
    \Delta U_{N,M}(\cdot, t)&:=& U_{N,M}(\cdot, t) - u^h(\cdot, t).
\end{eqnarray*}
Then
\begin{eqnarray*}
&& \sum_{k=1}^M \int_{t_{k-1}}^{t_k}\int_\Omega \frac{U_{N,M}(\cdot, t_k)- U_{N,M}(\cdot,t_{k-1})}{\Dt}(v_N(\cdot, t) - U_{N,M}(\cdot, t_k))\cr
&=& \sum_{k=1}^M \int_{t_{k-1}}^{t_k}\int_\Omega \frac{u^h(\cdot, t_k)- u^h(\cdot,t_{k-1})}{\Dt}(v^h_\epsilon(\cdot, t) -
u^h(\cdot,t_k))dx  + \mbox{Err}_2,
%&& + \sum_{k=1}^M \int_{t_{k-1}}^{t_k}\int_\Omega \frac{\Delta U_{N,M}(\cdot, t_k) - \Delta U_{N,M}(\cdot, t_{k-1})}{\Delta t}(\Delta v_N(\cdot, t) - \Delta U_{N,M}(\cdot, t_k))
\end{eqnarray*}
where $\hbox{Err}_2$ can be written as
\begin{align*}
\mbox{Err}_2
&=\sum_{k=1}^M \int_{t_{k-1}}^{t_k}\int_\Omega \frac{\Delta U_{N,M}(\cdot, t_k) - \Delta U_{N,M}(\cdot, t_{k-1})}{\Delta t}(v^h_\epsilon(\cdot, t) -
u^h(\cdot, t_k)) \\
&\qquad + \sum_{k=1}^M \int_{t_{k-1}}^{t_k}\int_\Omega \frac{u^h(\cdot, t_k) - u^h(\cdot, t_{k-1})}{\Delta t}(\Delta v_N(\cdot, t) -
 \Delta U_{N,M}(\cdot, t_k)) \\
&\qquad + \sum_{k=1}^M \int_{t_{k-1}}^{t_k}\int_\Omega \frac{\Delta U_{N,M}(\cdot, t_k) - \Delta U_{N,M}(\cdot,
t_{k-1})}{\Delta t}(\Delta v_N(\cdot, t) - \Delta U_{N,M}(\cdot, t_k)).
\end{align*}

The three terms in $\hbox{Err}_2$ can be bounded in a similar fashion. We only give the details of the bounds for the first and second terms.
The third term can be bounded similarly. We first point out the following facts,
$\|v^h_\epsilon\|, \left\|u^h\right\| \le C$ that can be easily proved with Lemma~\ref{FDstable}.
Note that by Lemma~\ref{useful}
\begin{equation*}
\|\Delta U_{N,M}\|_{L^1([0,T];L^2(\Omega))} \le CT(\|u^0\|_{\opLip(\alpha, L^2(\Omega))} + \|f\|_{\opLip(\alpha, L^2(\Omega))})h^\alpha.
\end{equation*}
By using Cauchy-Schwarz inequality, the first term in $\hbox{Err}_2$ can be bounded by
$$
\frac{2}{\Delta t} \|\Delta U_{N,M}\|_{L^2([0,T];L^2(\Omega))} (\|v^h_\epsilon\| + \left\|u^h\right\|)
\le CT  (\|u^0\|_{\opLip(\alpha, L^2(\Omega))} + \|f\|_{\opLip(\alpha, L^2(\Omega))}) \frac{h^\alpha}{\Delta t}.
$$
Next we look at the second term in $\hbox{Err}_2$.
\begin{eqnarray*}
    \|\Delta v_N(\cdot, t)\|_{L^2(\Omega)} &=& \|\mathcal{I}^h v_\epsilon(\cdot, t) - \Ph v_\epsilon(\cdot, t)\|_{L^2(\Omega)}\\
    &\le& \|\mathcal{I}^h v_\epsilon - v_\epsilon\|_{L^2(\Omega)} + \|v_\epsilon - \Ph v_\epsilon\|_{L^2(\Omega)}\\
    &\le& C\|\mathcal{I}^h v_\epsilon - v_\epsilon\|_{W^{1,1}(\Omega)} + Ch|v_\epsilon|_{W^{2,1}(\Omega)} \\
    &\le& C\|\mathcal{I}^h v_\epsilon - v_\epsilon\|_{W^{1,1}(\Omega)} + C\frac{h}{\epsilon}|v_\epsilon|_{W^{1,1}(\Omega)}
\le C h^\alpha \|v(\cdot, t)\|_{W^{1,1}(\Omega)}
\end{eqnarray*}
by using (\ref{v_epsilon-int-v-bound})(and recall that $\epsilon = h^{1-\alpha}$).
%where we have applied Sobolev embedding inequality (or the isoperimetric inequality) to the two terms, respectively.

%We bound these two terms using (\ref{v_epsilon-int-v-bound}) and~(\ref{v_epsilon-w11-bound}) to have
%\begin{equation*}%
%   \|\Delta v_N(\cdot, t)\|_{L^2(\Omega)} \le Ch^\alpha|v(\cdot, t)|_{W^{1,1}(\Omega)}.
%\end{equation*}
Then  the second term  in $\hbox{Err}_2$ is bounded by
\begin{eqnarray*}
    && \sum_{k=1}^M\int_{t_{k-1}}^{t_k} \int_\Omega\frac{u^h(\cdot, t_k)-u^h(\cdot, t_{k-1})}{\Delta t} (\Delta v_N(\cdot, t) - \Delta U_{N,M}(\cdot, t_k))\,dxdt \\
    &\le& \sum_{k=1}^M\int_{t_{k-1}}^{t_k}C\left\|\frac{d}{dt}U_{N,M}\right\|_{L^2(\Omega)}\left\|\Delta v_N(\cdot, t) - \Delta U_{N,M}(\cdot, t_k)\right\|_{L^2(\Omega)}\,dt\\
    &\le& C\left(\|\Delta v_N\|_{L^1([0,T];L^2(\Omega))} + \|\Delta U_{N,M}\|_{L^1([0,T];L^2(\Omega))}\right)\\
%   &\le& \sum_{k=1}^M\int_{t_{k-1}}^{t_k} \int_\Omega \frac{C}{\Delta t}C\sqrt{T}(\|u^0\|_{\opLip(\alpha, L^2)} + \|f\|_{\opLip(\alpha, L^2)} + |v|_{W^{1,1}})h^p \\
    &\le& CT(\|u_0\|_{\opLip(\alpha, L^2)} + \|f\|_{\opLip(\alpha, L^2)} + \|v\|_{L^1([0, T]; W^{1,1}(\Omega))})h^\alpha,
\end{eqnarray*}
where we have used Lemmas~\ref{FDTbound}, \ref{useful} and \ref{approximation}.
We also bound the other two terms with the order of $h$ being $1$ and $1+\alpha$ respectively.
Consuming all higher orders of $h$, the left side of~(\ref{keyineq21}) can be bounded from below by
\begin{equation*}
\sum_{k=1}^M \int_{t_{k-1}}^{t_k}\sum_{i,j}\frac{u^k_{i,j}-u^{k-1}_{i,j}}{\Dt}(v^h_{\epsilon,i,j}- u^k_{i,j})\,h^2 -
C(\|u_0\|_{\opLip(\alpha, L^2)} + \|f\|_{\opLip(\alpha,L^2)} + \|v\|_{L^1([0,T];W^{1,1}(\Omega))}) Th^\alpha.
\end{equation*}
We sum up our bound on~(\ref{temp-1}) as
\begin{align}
\label{key-bound-1}
&\phantom{ {}={}} \int_0^T\int_\Omega \od U_{N,M} (v_N(\cdot,t)-U_{N,M}(\cdot.t))\,dxdt \nonumber \\
&\ge \sum_{k=1}^M \int_{t_{k-1}}^{t_k}\sum_{i,j}\frac{u^k_{i,j}-u^{k-1}_{i,j}}{\Dt}(v^h_{\epsilon, i,j}- u^k_{i,j})\,h^2 \nonumber \\
&\qquad -C(\|u_0\|_{\opLip(\alpha, L^2)} + \|f\|_{\opLip(\alpha,L^2)} + \|v\|_{L^1([0,T];W^{1,1}(\Omega))})T
\frac{h^\alpha}{\Delta t} - C\Delta t.
\end{align}

We next bound the second term of the left-hand side of~(\ref{keyineq21})(the variation term),
\begin{align*}
%&\phantom{ {}={} }
\int_0^T J(v_N(\cdot,t)) - J(U_{N,M}(\cdot,t))\,dt
&= \sum_{k=1}^M \int_{t_{k-1}}^{t_k} J(v_N(\cdot,t))  - J(U_{N,M}(\cdot,t))\,dt\\
&= \sum_{k=1}^M \int_{t_{k-1}}^{t_k} J^h(v^h_\epsilon(t)) - J^h(u^h(t_k))\,dt + \hbox{Err}_3
\end{align*}
with
\begin{align*}
\hbox{Err}_3
&= \sum_{k=1}^M \int_{t_{k-1}}^{t_k} J(v_N(\cdot, t)) - J^h(v^h_\epsilon(t))\,dt -
\sum_{k=1}^M\int_{t_{k-1}}^{t_k}J^h(u^h(t)) - J^h(u^h(t_k))\,dt - \cr
&\qquad \sum_{k=1}^M\int_{t_{k-1}}^{t_k} J(U_{N,M}(\cdot, t)) - J^h(u^h(t))\,dt
\end{align*}
By Lemma~\ref{J-difference}, the first and the third term can be bounded by $C_1h^\alpha T$ and $C_2h^\alpha T$ respectively. To
bound the second term we use  the convexity of $J^h$ and the monotonicity of $J^h$ shown in Lemma~\ref{variation-monotonicity-lemma},
\begin{align*}
&\phantom{ {}={}} \left| \sum_{k=1}^M\int_{t_{k-1}}^{t_k}J^h(u^h(t)) - J^h(u^h(t_k))\,dt  \right|  \\
&\leq \sum_{k=1}^M \int_{t_{k-1}}^{t_k}\left|\frac{t-t_{k-1}}{\Dt} J^h(u^h(t_k))
+ \frac{t_k - t}{\Dt} (J^h(u^h(t_{k-1})) - J^h(u^h(t_k)))\right|\,dt \\
&= \sum_{k=1}^M \left|J^h(u^h(t_{k-1}))- J^h(u^h(t_k))\right| \int_{t_{k-1}}^{t_k} \frac{t_k - t}\Dt \,dt
\leq \sum_{k=1}^M 2Ch^\alpha \Dt = CTh^\alpha ,
%&\leq \Dt J(U_{N,M}(\cdot,0)) = \Dt J(u^0).
\end{align*}
where we have used Lemma~\ref{J-difference}.

Collecting the results together, we have
\begin{align}
\label{key-bound-2}
\phantom{ {}={} }\int_0^T J(v_N(\cdot,t)) - J(U_{N,M}(\cdot,t))\,dt
\ge \sum_{k=1}^M \int_{t_{k-1}}^{t_k} J^h(v^h_\epsilon(\cdot,t)) - J^h(u^h(t_k))\,dt -C h^\alpha T .
\end{align}
Put all the bounds~(\ref{key-bound-1}) and~(\ref{key-bound-2}) together, we have
\begin{align*}
&\phantom{ {}={} }\int_0^T\int_\Omega \od U_{N,M} (v_N(\cdot,t)-U_{N,M}(\cdot.t))\,dxdt + \int_0^T J(v_N(\cdot,t)) - J(U_{N,M}(\cdot,t))\,dt\\
&\ge \sum_{k=1}^M \int_{t_{k-1}}^{t_k}\left\{\sum_{i,j}\frac{u^k_{i,j}-u^{k-1}_{i,j}}{\Dt}(v_{i,j}- u^k_{i,j})\,h^2
+ J^h(v^h_\epsilon(\cdot,t)) - J^h(u^h(t_k))\,dt\right\}\\
&\qquad -C(\|u_0\|_{\opLip(\alpha, L^2)} + \|f\|_{\opLip(\alpha,L^2)} + \|v\|_{L^1([0,T];W^{1,1}(\Omega))}) T
\frac{h^\alpha}{\Delta t} - C\Delta t - Ch^\alpha T.
\end{align*}
Using Lemma~\ref{LTidea} for the first term on the right-hand side of the inequality above,
we let $h$, $\Delta t$ tend to zero in the fashion (\ref{laimj}) to obtain  the desired result.
\end{proof}

Finally we are ready to prove the main result of this section.
\begin{theorem}
\label{mainFDconv}
Suppose that $u_0 \in W^{1,1}(\Omega), f \in \opLip(\alpha, L^2(\Omega))$.
There exists a function $U^*$ in $L^2(\Omega_T)$ so that $U_{N,M}$ converge to $U^*$ weakly as $N, M \to \infty$ in the fashion
(\ref{laimj}) and $U^*$ is the weak solution of (\ref{epsTEROF2}).
\end{theorem}
\begin{proof}
By Lemma~\ref{FDL2bound}, there exists a weakly convergent subsequence of $\{U_{N,M},  N\ge 1, M\ge 1\}$
in $L^2(\Omega_T)$. For convenience, we assume the whole sequence
converges to $U^*\in L^2(\Omega_T)$ weakly.
%By the weakly lower semicontinuity of the BV seminorm,
%$$
%|U^*|_{BV(\Omega)}\le \liminf^h_{N\to \infty\atop M\to \infty} |U_{N,M}|_{BV(\Omega)}
%$$.
We now show $U^*$ is the weak solution of the gradient flow as in Definition~\ref{weaksol}. As the weak solution is
unique, the whole sequence $\{U_{N,M}, N\ge 1, M\ge 1\}$ converges weakly to $U^*$.

By using Theorem~\ref{FengThm1}, we need to show that $U^*$ satisfies the following inequality:
\begin{eqnarray}
\label{FengXiaobing}
&&\int_0^s \int_\Omega \od v (v- U^*)dxdt +
\int_0^s (J(v)- J(U^*)) dt \cr
&\ge & \frac{1}{2}\left[ \int_\Omega (v(x, s)- U^*(x, s))^2dx -
\int_\Omega (v(x,0)- u_0(x,0))^2dx\right]
\end{eqnarray}
for all $v \in L^1([0, T], W^{1,1}(\Omega))$ with ${\partial
\over \partial {\bf n}}v(x,t) =0$ for all $(t,x)\in [0, T)\times
\partial \Omega$, where
$$
J(u) =\int_\Omega \sqrt{\epsilon+|\nabla u(x,t)|^2}dx + \frac{1}{2\lambda}
\int_\Omega |f(x,t)- u(x,t)|^2dx.
$$
By the lower semi-continuity of $J$, Fatou's lemma and standard weak convergence, we have
\begin{eqnarray}
\label{firsteq}
&&\int_0^s \int_\Omega \od v (v- U^*)dxdt +
\int_0^s (J(v)- J(U^*)) dt \cr
&\ge & \liminf_{N, M\to \infty}\left[
\int_0^s \int_\Omega \od v (v- U_{N,M})dxdt +
\int_0^s (J(v)- J(U_{N,M})) dt\right].
\end{eqnarray}
By the weak lower semi-continuity of the $L^2$ norm
\begin{eqnarray}
\label{secondeq}
&& \liminf_{N,M\to \infty} \frac{1}{2}\left[ \int_\Omega (v(x, s)- U_{N,M}(x, s))^2dx -
\int_\Omega (v(x,0)- u_0(x,0))^2dx\right]  \cr
&\ge & \frac{1}{2}\left[ \int_\Omega (v(x, s)- U^*(x, s))^2dx -
\int_\Omega (v(x,0)- u_0(x,0))^2dx\right].
\end{eqnarray}
%Indeed, in (\ref{secondeq}), $v(x,s)- U_{N,M}(x,s) $ is convergent weakly to $v(x,s)- U^*(x,s)$ in $L^2(\Omega_T)$ and
%for  all $s\in [0, T]$, $v(x,s)- U_{N,M}(x,s) $ is convergent weakly to $v(x,s)- U^*(x,s)$ in $L^2(\Omega)$. They define  linear functionals
%on $L^2(\Omega)$ for all $s\in [0, T]$.  By the Banach-Steinhaus theorem, the norm of the linear function satisfies the following inequality
%$$\int_\Omega (v(x, s)- U^*(x, s))^2dx \le \liminf_{N,M\to \infty} \int_\Omega (v(x, s)- U_{N,M}(x, s))^2dx $$
%for almost all $s\in [0, T]$.

We now prove the following inequality to finish the proof.
\begin{align*}
& \int_0^s \int_\Omega \od v (v- U_{N,M})dxdt +
\int_0^s (J(v)- J(U_{N,M})) dt \cr
& \ge
\frac{1}{2}\left[ \int_\Omega (v(x, s)- U_{N,M}(x, s))^2dx -
\int_\Omega (v(x,0)- u_0(x,0))^2dx\right] - \hbox{Error}_{N,M}
\end{align*}
where $\hbox{Error}_{N,M}>0$ is an error term that goes to zero as $N, M\to \infty$.
It's straightforward to verify(cf. \cite{FP03}) that the above inequality is equivalent to
\begin{equation}
\label{entropy-inequality2}
\int_0^s \int_\Omega \od U_{N,M} (v- U_{N,M})dxdt +
\int_0^s (J(v)- J(U_{N,M})) dt
\ge -\hbox{Error}_{N,M}.
\end{equation}

By Proposition~\ref{proposition1}, there exits a sequence $\{v_N\}$, so that
\begin{equation*}
    \lim_{N\to\infty} v_N = v   \qquad \mbox{in $L^1([0,T];W^{1,1}(\Omega))$,.}
\end{equation*}
and
\begin{equation*}
\int_0^s \left[\int_\Omega \od U_{N,M} (v_N- U_{N,M})dx + (J(v_N)- J(U_{N,M}))\right] dt
\ge  -\hbox{Err}_{N,M}
\end{equation*}
where $\mbox{Err}_{N,M}$ only depends on $f$ and $v$, and tends to zero as $N,M$ tend to infinity. We replace the original $W^{1,1}$ test function
 $v(\cdot, t)$ in~(\ref{entropy-inequality2}) by $v_N$ that is in $L^1([0, T], S^0_1(\Delta_N))$, therefore introduces an error $e_{N,M}$.
% \textcolor[rgb]{1,0,0}{ Theorem 4.4.20 in \cite{BS94} only take care of the $L_2$ term in $J(v)$. That is, the $L_2$ in $J(v_N)$ converges to
% the $L_2$ term in $J(v)$. It does not prove that the BV term in $J(v_N)$ converges to the BV term in $J(v)$ if we only assume that
%$v\in W^{1,1}(\Omega)$.  }

%{\color{blue}You are right on $J(v_N)\to J(v)$, We needs more regularies here\dots}
$$
e_{N,M} = \int_0^s \int_\Omega\od U_{N,M}(v-v_N) + J(v) - J(v_N).
$$
It is easy to show $e_{N,M}$ tends to zero as $N, M$ go to infinity by
%standard density argumentation based on linear finite element approximation property (cf. Theorem 4.4.20 in \cite{BS94}) and
Lemmas~\ref{FDTbound} and ~\ref{spline-approximation}. Thus we complete the proof.
\end{proof}

\section{Numerical Solution of Our Finite Difference Scheme}
The system (\ref{FD2}) of nonlinear equations has been solved by many methods as explained in \cite{VO96}. In \cite{DV97}, the researchers
provided an analysis of a fixed point method proposed in \cite{VO96} based on auxiliary variable and functionals and
proved that the iterative method converges.
In this section, we mainly present another method to show the convergence of the fixed point method. From notation simplicity,
we assume the grid size $h = 1$ in this section that has no influence in the convergence analysis of our algorithm.

First of all, let us explain the fixed point method. Recall that we need to solve
$\{u^k_{i,j}, 0\le i, j\le N-1\}$ from the following equations
\begin{align*}
\frac{u^k_{i,j}- u^{k-1}_{i,j}}\Dt & -
\frac{1}{2}\hbox{ div}^+ \left( \dfrac{\nabla^+ u^k_{i,j}}{\sqrt{\epsilon + |\nabla^+ u^k_{i,j}|^2}}\right)
-\frac{1}{2}\hbox{ div}^- \left( \dfrac{\nabla^- u^k_{i,j}}{\sqrt{\epsilon + |\nabla^- u^k_{i,j}|^2}} \right)  \nonumber\\
&\qquad+ \frac{1}{\lambda}(u^k_{i,j}- f^h_{i,j}) = 0,\quad   0\le i, j\le N-1,
\end{align*}
assuming that we have the solution $\{u^{k-1}_{i,j}, 0\le i, j\le N-1\}$. Let us define an iterative algorithm to compute
$u^k_{i,j}$.
\begin{algorithm}
\label{FDalgo}
Starting with $v^0_{i,j}= u^{k-1}_{i,j}, 0\le i, j\le N-1$, for $\ell=1, 2, \cdots, $, we compute array $\{v^\ell_{i,j}, 0 \le i,j \le N-1\}$ by
\begin{align}
\label{FDiter}
\frac{v^{\ell}_{i,j}- u^{k-1}_{i,j}}\Dt &=
\frac{1}{2}\divp \left( \dfrac{\nabla^+ v^\ell_{i,j}}{\sqrt{\epsilon + |\nabla^+ v^{\ell-1}_{i,j}|^2}}
\right)
+\frac{1}{2}\divm \left( \dfrac{\nabla^- v^\ell_{i,j}}{\sqrt{\epsilon + |\nabla^- v^{\ell-1}_{i,j}|^2}} \right)  \nonumber\\
&\qquad- \frac{1}{\lambda}(v^\ell_{i,j}- f^h_{i,j}), \quad   0\le i, j\le N-1,
\end{align}
together with boundary conditions in (\ref{FD2}).
\end{algorithm}

We now show that the iterative solutions $\{v^\ell_{i,j}, 0\le i, j\le N-1\}, \ell\ge 0$ converge. Indeed, we first have
\begin{lemma}
There exists a positive constant $C$ dependent only on $f$ and initial values $u^{k-1}_{i,j}$ such that
\begin{equation}
\label{FDiter2}
\|v^\ell\|^2 := \sum_{i,j} |v^\ell_{i,j}|^2 \le C
\end{equation}
for all $\ell\ge 1$.
\end{lemma}
\begin{proof}
Multiplying $v^\ell_{i,j}$ to the equation (\ref{FDiter}) and summing over $i, j=0, \cdots, N-1$, we have
\begin{eqnarray*}
\frac{\|v^\ell\|^2}\Dt &=& \frac{1}\Dt\sum_{i,j} u^{k-1}_{i,j}v^\ell_{i,j} -\frac{1}{2} \sum_{i,j} \dfrac{\nabla^+ v^\ell_{i,j}
\nabla^+ v^\ell_{i,j}}{\sqrt{\epsilon + |\nabla^+ v^{\ell-1}_{i,j}|^2}} \cr
&& - \frac{1}{2} \sum_{i,j} \dfrac{\nabla^- v^\ell_{i,j} \nabla^- v^\ell_{i,j}}
{\sqrt{\epsilon + |\nabla^- v^{\ell-1}_{i,j}|^2}} - \frac{1}{\lambda}\|v^\ell\|^2 + \frac{1}{\lambda}\sum_{i,j}f^h_{i,j}v^\ell_{i,j}.
\end{eqnarray*}
By using the Cauchy-Schwarz equality, it follows that
\begin{eqnarray*}
(\frac{1}\Dt+\frac{1}{\lambda}) \|v^\ell\|^2 \le \frac{1}\Dt \|u^{k-1}_{i,j}\|\|v^\ell\| + \frac{1}{\lambda}\|f^h\|\|v^\ell\|.
\end{eqnarray*}
Hence, $\|v^\ell\|$ is bounded by a constant $C$ independent of $\ell$.
\end{proof}

It follows that the sequence of  vectors $\{v^\ell_{i,j}, 0\le i, j\le N-1\}, \ell\ge 1$ contains a convergent subsequence.
Let us say the vectors
$v^{\ell_k}_{i,j}, 0\le i, j\le N-1$ converge to $v^*_{i,j}, 0\le i, j\le N-1$. Next we claim that the whole sequence converges.
To prove this claim, we recall the energy functional
\begin{align}
\label{FDenergy2}
E^h(v) = J^h(v) + \frac{1}{2\Dt}\sum_{i,j}(v_{i,j}-u^{k-1}_{i,j})^2.
\end{align}
where
\begin{align}
\label{FDenergy}
J^h(v) =
\frac{1}{2}\sum_{i,j}\sqrt{\epsilon + |\nabla^+ v_{i,j}|^2} +
\frac{1}{2}\sum_{i,j}\sqrt{\epsilon + |\nabla^- v_{i,j}|^2}
+\frac{1}{2\lambda}\sum_{i,j}( v_{i,j} - f^h_{i,j})^2.
\end{align}

Let us prove the following lemma
\begin{lemma}
\label{FDiterkey}
Given $v^\ell$ defined in Algorithm~\ref{FDalgo}, we have for all $\ell\ge 1$
$$
\frac{1}{2\lambda} \|v^\ell - v^{\ell-1}\|^2 \le E(v^{\ell-1})- E(v^\ell).
$$
\end{lemma}
\begin{proof}
Fix $\ell\ge 1$.
For the terms in $E(v^{\ell-1})- E(v^\ell)$, we first consider
\begin{eqnarray}
\label{FDstep1}
&&\frac{1}{2\Dt}\sum_{i,j}(v^{\ell-1}_{i,j}-u^{k-1}_{i,j})^2 - \frac{1}{2\Dt}\sum_{i,j}(v^\ell_{i,j}-u^{k-1}_{i,j})^2\cr
&=& \frac{1}{2\Dt}\sum_{i,j}(v^{\ell-1}_{i,j}-v^\ell_{i,j})^2 + \frac{1}\Dt \sum_{i,j}(v^\ell_{i,j}-u^{k-1}_{i,j})(v^{\ell-1}_{i,j}-v^\ell_{i,j}).
\end{eqnarray}
To estimate the second term on the right-hand side of the equation above, we
multiply $v^{\ell-1}_{i,j}-v^\ell_{i,j}$ to the equation (\ref{FDiter})
and sum over $i, j=0, \cdots, N-1$ to have
\begin{eqnarray*}
&& \frac{1}\Dt\sum_{i,j}(v^\ell_{i,j}-u^{k-1}_{i,j})(v^{\ell-1}_{i,j}-v^\ell_{i,j}) \\
&=& -\frac{1}{2} \sum_{i,j} \dfrac{\nabla^+ v^\ell_{i,j} \nabla^+ (v^{\ell-1}_{i,j}-v^\ell_{i,j})  }
{\sqrt{\epsilon + |\nabla^+ v^{\ell-1}_{i,j}|^2}}
 - \frac{1}{2} \sum_{i,j} \dfrac{\nabla^- v^\ell_{i,j} \nabla^- (v^{\ell-1}_{i,j}-v^\ell_{i,j})  }
{\sqrt{\epsilon + |\nabla^- v^{\ell-1}_{i,j}|^2}}
 - \frac{1}{\lambda}\sum_{i,j}(v^\ell_{i,j}-f^h_{i,j}) (v^{\ell-1}_{i,j}-v^\ell_{i,j}) .
\end{eqnarray*}
Using an elementary inequality $a(b-a) \le b^2/2- a^2/2$, we can easily see
\begin{equation}
-\frac{1}{2} \sum_{i,j} \dfrac{\nabla^+ v^\ell_{i,j} \nabla^+ (v^{\ell-1}_{i,j}-v^\ell_{i,j})  }
{\sqrt{\epsilon + |\nabla^+ v^{\ell-1}_{i,j}|^2}}
\ge  -\frac{1}{4}\sum_{i,j} \dfrac{\nabla^+ v^{\ell-1}_{i,j} \nabla^+ v^{\ell-1}_{i,j}}
{\sqrt{\epsilon + |\nabla^+ v^{\ell-1}_{i,j}|^2}}  +  \frac{1}{4} \sum_{i,j} \dfrac{\nabla^+ v^{\ell}_{i,j} \nabla^+ v^{\ell}_{i,j}}
{\sqrt{\epsilon + |\nabla^+ v^{\ell-1}_{i,j}|^2}}.
\end{equation}
Similar for other term involving $\nabla^-$.

Next we have
\begin{eqnarray}
\label{FDstep2}
&& \frac{1}{2\lambda}\sum_{i,j}( v^{\ell-1}_{i,j} - f^h_{i,j})^2 - \frac{1}{2\lambda}\sum_{i,j}( v_{i,j}^\ell - f^h_{i,j})^2\cr
&=& \frac{1}{2\lambda}\sum_{i,j} (v^{\ell-1}_{i,j}-v^\ell_{i,j})(v^{\ell-1}_{i,j}+v^\ell_{i,j}-2f^h_{i,j})\cr
&=& \frac{1}{2\lambda}\sum_{i,j} (v^{\ell-1}_{i,j}-v^\ell_{i,j})^2 +
\frac{1}{\lambda}\sum_{i,j} (v^\ell_{i,j}-f^h_{i,j})(v^{\ell-1}_{i,j}-v^\ell_{i,j}).
\end{eqnarray}

Finally we need another elementary inequality:
for any real numbers $a, b$ and $\epsilon>0$,
$$
2\sqrt{\epsilon+ b^2} - 2\sqrt{\epsilon+ a^2} \ge \frac{ b^2}{\sqrt{\epsilon+ b^2} } - \frac{a^2}{\sqrt{\epsilon+ b^2}}.
$$
This inequality can be proved as follows.
By the arithmetic-geometric inequality, we have
$$
2\sqrt{\epsilon+ a^2} \sqrt{\epsilon+ b^2} \le 2\epsilon+ a^2 + b^2.
$$ Rearranging the terms, we get
$$
b^2- a^2 \le 2(\epsilon+ b^2) - 2\sqrt{\epsilon + a^2} \sqrt{\epsilon + b^2}.
$$
Now dividing $\sqrt{\epsilon+ b^2}$ both sides, we obtain the desired inequality.

Using the above inequality, we can easily verify the following inequality
\begin{equation}
\label{FDstep3}
\frac{1}{2}\sum_{i,j}\sqrt{\epsilon + |\nabla^+ v^{\ell-1}_{i,j}|^2}  -\frac{1}{2}\sum_{i,j}\sqrt{\epsilon + |\nabla^+ v^\ell_{i,j}|^2}
\ge \frac{1}{4}\sum_{i,j}  \dfrac{\nabla^+ v^{\ell-1}_{i,j} \nabla^+ v^{\ell-1}_{i,j}}
{\sqrt{\epsilon + |\nabla^+ v^{\ell-1}_{i,j}|^2}} - \frac{1}{4}\sum_{i,j}  \dfrac{\nabla^+ v^{\ell}_{i,j} \nabla^+ v^{\ell}_{i,j}}
{\sqrt{\epsilon + |\nabla^+ v^{\ell-1}_{i,j}|^2}}.
\end{equation}

Similar for the terms involving $\nabla^-$. We now add all equalities and inequalities (\ref{FDstep1}), (\ref{FDstep2}) and (\ref{FDstep3})
 together to have
\begin{equation}
E(v^{\ell-1})- E(v^\ell)\ge
\frac{1}{2\lambda}\sum_{i,j} (v^{\ell-1}_{i,j}-v^\ell_{i,j})^2 .
\end{equation}
This completes the proof.
\end{proof}

We are now ready to prove the main result in this subsection.
\begin{theorem}
\label{FDmain2}
The iterative solutions defined in Algorithm~\ref{FDalgo} converge to the solution of (\ref{FD2}) for any fixed $k\ge 1$.
\end{theorem}
\begin{proof}
We have already shown that the iterative solution vectors $\{v^\ell_{i,j}, 0\le i, j \le N-1\}$ have a convergent subsequence
$\{v^{\ell_k}_{i,j}, 0\le i, j\le N-1\}, k=1, 2, \cdots$ to a vector $v^*$. It is easy to see that the energies
$E(v^{\ell_k}), k\ge 1$ are also convergent to $E(v^*)$.
By Lemma~\ref{FDiterkey}, we know that energies $E(v^\ell)$ are decreasing for all $\ell$ and hence, $E(v^{\ell_k+1})$
decrease to $E(v^*)$. By using Lemma~\ref{FDiterkey} again, we see $\|v^{\ell_k+1}-v^{\ell_k}\|^2\le 2\lambda (E(v^{\ell_k}
-E(v^{\ell_k+1})\to 0$. Thus, $v^{\ell_k+1}, k\ge 1$ are also convergent to $v^*$.
The uniqueness of the solution of (\ref{FD2}) implies that $v^*$ is the solution vector $\{u^k_{i,j}, 0\le i, j\le N-1\}$.
\end{proof}

%We end this section with some computational results as follows.

\section{Computational Results}
We have implemented our iterative algorithm in the previous section in MATLAB. Let us report one numerical
example for simplicity.

\begin{example}
In this Example, we use the algorithm to remove the noised from
images. For comparison, we also provide denoised images by using a
standard Perona-Malik PDE method with diffusivity function
$c(s)=1/\sqrt{1+s}$ (cf. \cite{PM90}). 
%in Figure~\ref{fig:3} numerical resluts on noised image using the proposed algorithm.
A Gaussian noise with $\sigma^2=20$ is added to the clean image of LENA and BARBARA. The PSNR of the noised images is 22.11.
PSNR of the recovered images are shown on the top of the images.
The two denoised images are shown in Figures~\ref{fig:3} and \ref{fig:4}. The left one is done by the PM method and the right one is
based on our finite difference scheme.  From these examples, we can see that our finite difference scheme works as the same or slightly
better  than the Perona-Malik method.

%\vfill\eject
\begin{figure}[htbp]
%\centering
\begin{tabular}{rr}
\includegraphics[width=3in]{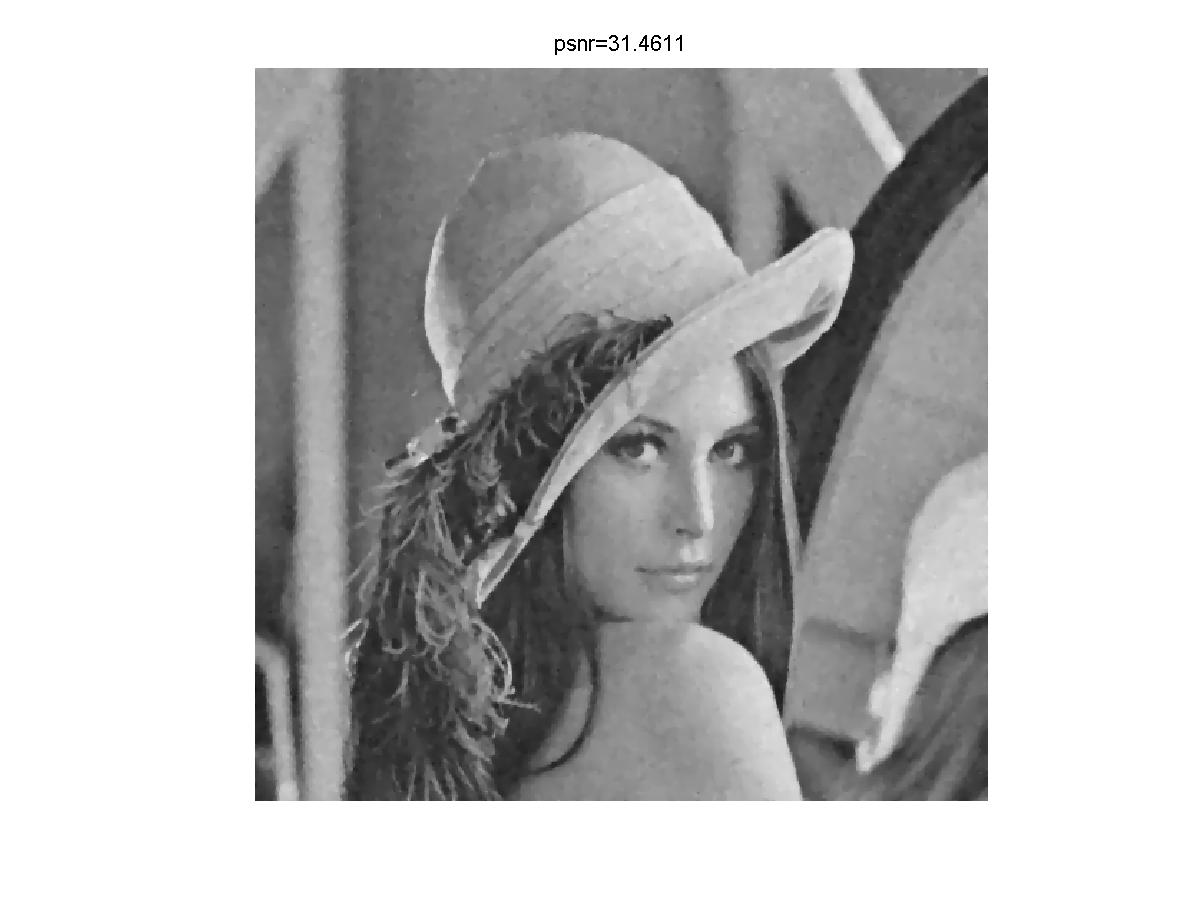}
& \includegraphics[width=3in]{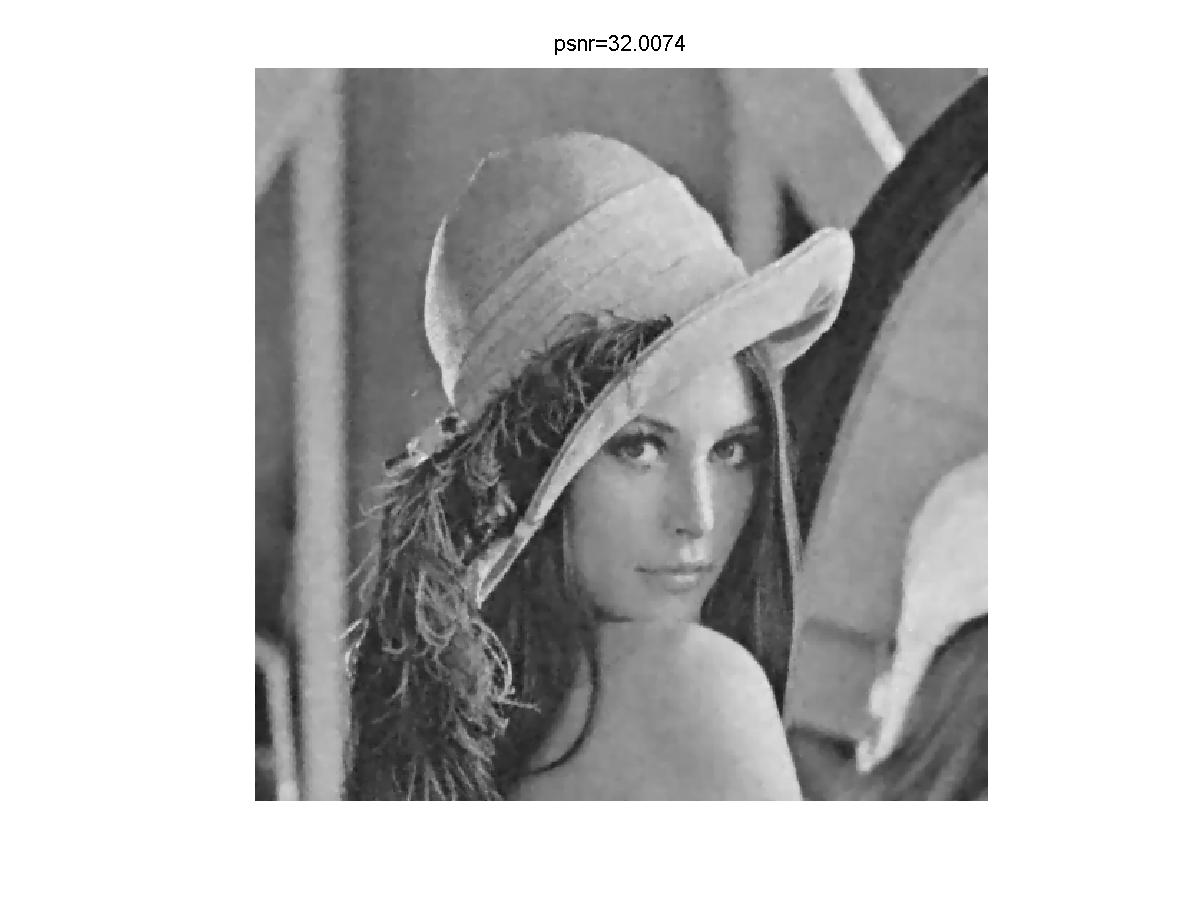}
\end{tabular}
\caption{The denoised  images by the PM method and the denoised image (right) by our finite difference scheme
\label{fig:3}}
\end{figure}

%\vfill\eject
\begin{figure}[htbp]
%\label{fig:}
%\centering
\begin{tabular}{ll}
\includegraphics[width=3in]{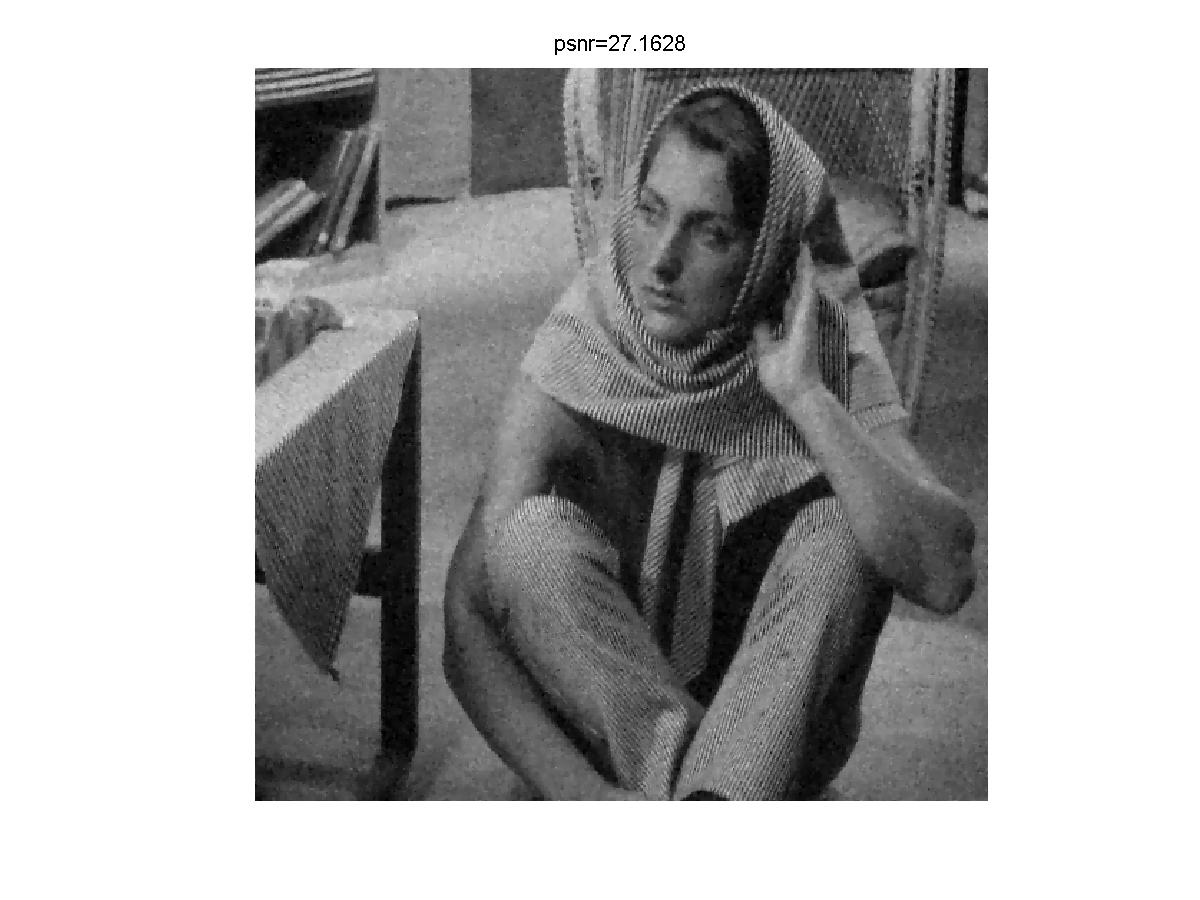}
& \includegraphics[width=3in]{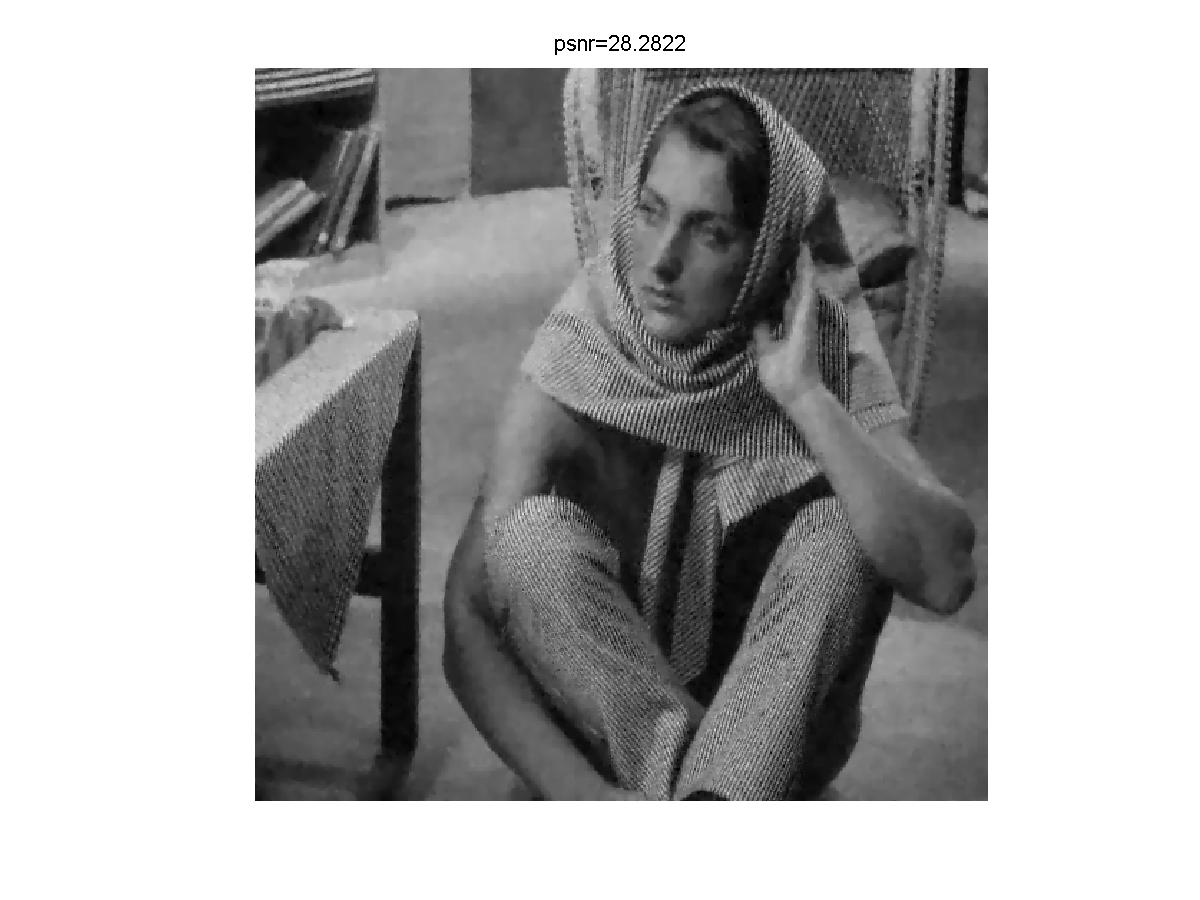}
\end{tabular}
\caption{The denoised  images by the PM method and the denoised image (right) by our finite difference scheme
\label{fig:4}}
\end{figure}

\end{example}

\begin{acknowledgement}
The authors would like to thank Leopold Matamba Messi for  several suggestions which improve
the readability of this paper.
\end{acknowledgement}


\begin{thebibliography}{99}
\bibitem{av} R. Acar and C.R. Vogel (1994): Analysis of bounded variation penalty methods for
ill-posed problems, Inverse Problems, 10, 1217--1229.

\bibitem{ABC01}
F. Andreu, C. Ballester, V. Caselles, and J. M. Maz\'on, The Dirichlet problem for the total
variation flow, J. Funct. Anal., 180(2001):347--403.

\bibitem{ABC01b}
F. Andreu, C. Ballester, V. Caselles, and J. M. Maz\'on, Minimizing total variation flow,
Differential Integral Equations, 14(2001):321--360.

\bibitem{ABC02}
F. Andreu, V. Caselles, J. I. D\'iaz, and J. M. Maz\'on, Some qualitative properties for the
total variation flow, J. Funct. Anal., 188(2002):516--547.

\bibitem{CDPX99}
A. Cohen, R. DeVore, P. Petrushev, H. Xu, Nonlinear approximation and the space of $BV(\mathbb{R}^2)$,
Amer. J. Math., 121 3(1999), 587-628

\bibitem{BS94}
S. Brenner and L. R. Scott, The Mathematical Theory of Finite Element Methods, Spring-Verlag, 1994.

\bibitem{DV97}
D. C. Dobson and C. R. Vogel. Convergence of an iterative method for total variation denoising.
SIAM J. Numer. Anal., 34(1997), 1779--1791.


\bibitem{FL10}
X. Feng and M. -J. Lai, a private communication, July, 2010.

\bibitem{FOP05}
X. Feng, M. von Oehsen, and A. Prohl. Rate of convergence of regularization procedures and
finite element approximations for the total variation flow, Numer. Math., 100(2005), 441--456.


\bibitem{FP03}
X. Feng and A. Prohl. Analysis of total variation flow and its finite element approximations,
Math. Mod. Num. Anal., 37(2003) 533--556.


\bibitem{FY09}
X. Feng and M. Yoon,  Finite element approximation of the gradient flow for a class of
linear growth energies with applications to color image denoising,
Int. J. Numer. Anal. Model.  6 (2009),  389--40.

\bibitem{Gerhardt80}
C. Gerhardt, Evolutionary surfaces of prescribed mean curvature, J. Diff. Eq. 36(1980), 139--172.

\bibitem{Giusti84}
E. Giusti, Minimal Surfaces and Functions of Bounded Variation, Birkhauser, 1984.

\bibitem{LLM12}
M. J. Lai and L. Matamba Messi, Piecewise Linear Approximation of
the continuous Rudin-Osher-Fatemi model for image denoising,
to appear in SIAM J. Num. Analysis,  2012.

\bibitem{LT78}
A. Lichnewsky and R. Temam, Pseudo-solution of the Time Dependent Minimal Surface Problem, J. of Differential Equations, 30(1978), 340--364.


%B. Lucier and J. Wang, Error bounds for finite-difference methods for Rudin-Osher-Fatemi image smoothing, UCLA Cam Technical Report 09-70.

\bibitem{PM90}
P. Parona and J. Malik, Scale-Space and Edge Detection Using
Anisotropic Diffusion, IEEE Trans. Pattern Analysis Machine Intelligence, 12 (1990), pp. 629--639.

\bibitem{VS02}
L. Vese and S. Osher, Numerical methods for p-harmonic
flows and applications to image processing, SIAM J. Numer. Anal., 40(2002), 2085--2104.


\bibitem{VO96}
C. R. Vogel and M. E. Oman, Iterative methods for total variation denoising, SIAM J. Sci. Comput., 17 (1996),  227-238.

\bibitem{LW10}
J. Wang and B. J. Lucier, Error bounds for finite-difference methods
for rudin-osher-fatemi image smoothing, SIAM J. Numerical Analysis,
49 (2011), pp. 845�-868.

\bibitem{Ziemer89}
W. P. Ziemer, Weakly differentiable functions, Springer-Verlag,
1989.
\end{thebibliography}
\end{document}